\let\ORIlabel\label
\let\ORIrefstepcounter\refstepcounter
\AddToHook{package/hyperref/before}{%
   \let\label\ORIlabel
   \let\refstepcounter\ORIrefstepcounter}
\documentclass{siamart220329}
\usepackage[
  letterpaper,
  left=1in,
  right=1in,
  top=0.9in,
  bottom=1in
]{geometry}
\usepackage{tabularx,booktabs,makecell}
\newcolumntype{Y}{>{\raggedright\arraybackslash}X}
\usepackage{graphicx}
\usepackage{bm}
\usepackage{comment}

\Crefname{equation}{}{}
\usepackage{xcolor}

\usepackage{amsmath,amssymb,fixmath}
\usepackage{mathtools}
\DeclareMathOperator*{\argmin}{arg\,min}

\DeclareMathOperator{\proj}{proj}
\DeclareMathOperator{\prox}{prox}
\DeclareMathOperator{\lnit}{lnit}
\DeclareMathOperator{\sig}{sigm}

\DeclareMathOperator{\CVaR}{CVaR}

\newcommand{\R}{\mathbb{R}}
\newcommand{\E}{\mathbb{E}}

\newcommand{\cX}{{\mathcal X}}

\newcommand{\ind}{{\bf{1}}} 

\def\cvar{{\sf CVaR}}
\def\var{{\sf VaR}}

\newtheorem{remark}{Remark}[section]
\crefname{remark}{remark}{remarks}
\newtheorem{example}{Example}[section]
\crefname{example}{example}{examples}

\usepackage[ruled,vlined,algo2e]{algorithm2e}
\crefname{equation}{}{}
\crefname{algocf}{algorithm}{algorithms}

\author{
	Will~Asness\thanks{\protect
        Division of Applied Mathematics,
        Brown University,
        Providence, RI 02912 USA
        (\email{will\_asness@brown.edu} and \email{brendan\_keith@brown.edu}).
    }
    \and
    Brendan~Keith\footnotemark[2]
    \and
    Boyan~Lazarov\thanks{\protect
         Center for Design Optimization,
         Lawrence Livermore National Laboratory,
         Livermore, CA 94550
         (\email{lazarov2@llnl.gov})
    }
    \and
    Anton~Malandii\footnotemark[2]\ \thanks{\protect
         Department of Applied Mathematics and Statistics,
         Stony Brook University,
         Stony Brook, NY 11794
         (\email{anton.malandii@stonybrook.edu} and \email{stanislav.uryasev@stonybrook.edu})
    }
    \and
    Stan~Uryasev\footnotemark[4]
}

\begin{document}

\title{
    Exponential Adaptive Smoothing and Importance sampling for optimization of the conditional value-at-risk
    \thanks{Submitted to the editors \today.
        \funding{This work was performed under the auspices of the U.S.~Department of Energy by Lawrence Livermore National Laboratory under Contract DE-AC52-07NA27344 and the LLNL-LDRD Program under Project tracking No.~25-ERD-030 and ~22-ERD-009. Release number LLNL-JRNL-2020035.
        WA, AM, and BK were supported in part by the U.S.\ Department of Energy, Office of Science Early Career Research Program under Award Number DE-SC0024335 and by the Center for Information Geometric Mechanics and Optimization (CIGMO), a PSAAP-IV Focused Investigatory Center funded by the U.S.\ Department of Energy, National Nuclear Security Administration under Award Number DE-NA0004261.
        BK was also supported in part by the Alfred P.\ Sloan Foundation via a Sloan Research Fellowship in Mathematics.
        }
    }
}

\maketitle

\begin{abstract}
    We present a novel method for solving conditional value-at-risk (CVaR) optimization problems based on the dual representation of CVaR, which is defined as the worst-case expectation over a risk envelope.
    The method is based on the Bregman proximal point algorithm and alternates between stochastic primal and dual stages.
    Every (inner) primal stage involves a subproblem solved by sampling from a probability distribution updated at each dual stage (outer iteration).
    The likelihood ratio of the dual probability distributions relative to the distribution underlying the original problem converges to the risk identifier of the solution's CVaR.
    Thus, the dual distribution provides the algorithm with a built-in importance sampling mechanism that draws from the tail of the underlying distribution.
    Because only samples in the tail influence the CVaR, and samples outside the tail are drawn with decreasing probability, the algorithm delivers exceptional performance over other stochastic approximation methods.
    We prove the convergence of the algorithm for convex objective functions.
    Our numerical experiments target representative problems in financial mathematics and machine learning,
    focusing on portfolio optimization and support-vector machines, 
    respectively.

\end{abstract}



\section{Introduction and motivation}

Decision-making under uncertainty arises across various disciplines, such as engineering design, defense planning, finance, and modern machine learning.
In such settings, uncertainty is commonly modeled probabilistically, and its impact on system performance is quantified via risk measures \cite{royset2024riskadaptiveapproachesstochasticoptimization,quadrangle, quadrangle2}.
Risk-averse optimization then seeks decisions that trade off average (risk-neutral) and worst-case (robust) performance across possible states (scenarios) \cite{tyrrell2015engineering,royset2024riskadaptiveapproachesstochasticoptimization,kodakkal2022risk}.

Low-probability, high-consequence events are often modeled by specifying a confidence level $\alpha \in (0,1)$ and optimizing the associated tail expectation; i.e., the conditional value-at-risk (CVaR) \cite{rockafellar2000optimization}.
Two difficulties are central in CVaR optimization.
First, CVaR is, in general, nonsmooth, which typically leads to subgradient-based algorithms \cite{Shor1998Nondifferentiable} whose convergence can be slow in the sense of modern complexity theory \cite{nesterov2018lectures}.
Second, when the underlying distribution is accessed only through sampling, stochastic approximation schemes may produce high-variance subgradient estimates; the issue becomes particularly severe as $\alpha \uparrow 1$, since most samples contribute nothing to the tail expectation.

\subsection{Related work on adaptive importance sampling for CVaR optimization}

The high variance of Monte Carlo and stochastic approximation estimators for tail-risk objectives has motivated a substantial line of work on importance sampling for VaR and CVaR estimation.
Bardou, Frikha, and Pagès~\cite{BardouFrikhaPages2009} proposed stochastic approximation schemes for estimating VaR and CVaR and combined them with adaptive unconstrained importance sampling.
Their approach targets the estimation problem and uses recursive updates of the change of measure to reduce the variance of tail estimators.
Related ideas also appear in risk-sensitive simulation and reinforcement learning.
Prashanth~\cite{Prashanth2014} developed policy-gradient methods for CVaR-constrained Markov decision processes and incorporated importance-sampling-based variance reduction, while Tamar, Glassner, and Mannor~\cite{TamarGlassnerMannor2015} studied sampling-based optimization of CVaR and likelihood-ratio estimators for CVaR gradients.

More recent work has further emphasized the difficulty of learning an efficient sampling distribution when the relevant tail region depends on an unknown optimizer.
Deo and Murthy~\cite{DeoMurthy2021} developed black-box importance-sampling procedures for VaR and CVaR estimation, where the sampling distribution is constructed adaptively from less rare samples.
He, Jiang, Lam, and Fu~\cite{HeJiangLamFu2024} studied adaptive importance sampling for stochastic root finding and quantile estimation, highlighting the circular dependence between the target solution and an effective change of measure.
Closest in spirit to the present work, Pieraccini and Vanzan~\cite{PieracciniVanzan2025} proposed an adaptive importance sampling algorithm for risk-averse optimization in which both the sample size and the sampling distribution are updated during the optimization process.

Our approach is complementary to this literature.
Rather than constructing an external change of measure for the underlying uncertainty distribution, we exploit the dual representation of CVaR~\cite{rockafellar2006generalized} and update a dual probability distribution over scenarios within the CVaR risk envelope.
This distribution serves a twofold purpose: it defines the adaptive smoothing of the primal subproblem and induces an adaptive sampling mechanism that concentrates computational effort on tail-relevant scenarios.
Thus, the sampling distribution is not learned as a separate variance-reduction device, but is coupled directly to the minimax structure of the CVaR objective.

Motivated by this perspective, this paper proposes a framework that addresses \emph{both} difficulties (nonsmoothness of the CVaR and high-variance of its subgradient estimates) simultaneously by working with a dual representation of CVaR.
At each outer iteration, we (i) solve an \emph{adaptively smoothed} stochastic optimization subproblem in the decision variable, and (ii) update an \emph{adaptive sampling distribution} that increasingly concentrates on tail scenarios.
Thus, the sampling distribution is learned hand-in-hand with the decision variable rather than prescribed in advance.

\subsection{Background}
Let $(\Omega, \mathcal{A}, \mathbb{P})$ be a discrete probability space, where $\Omega = \{\omega_1, \ldots, \omega_n\}$ denotes the sample space, $\mathcal{A} = 2^\Omega$ denotes the set of outcomes, and $\mathbb{P}$ denotes the probability measure with associated probability mass function $\bm{p} = (p_1, \ldots, p_n)$ having $p_i \coloneqq \mathbb{P}(\omega_i)$ for each $i =1,\ldots,n$.
We consider the stochastic optimization problem
\begin{equation}\label{eq:cvar_opt}
    \min_{\mathbold{x}\in\mathcal{X}}~\cvar_\alpha\bigl(F(\mathbold{x},\omega)\bigr),
\end{equation}
where $\mathcal{X}\subseteq\mathbb{R}^d$ is closed, convex, and nonempty, $F:\mathcal{X}\times\Omega\to\mathbb{R}$ is convex in $\mathbold{x}$ for each $\omega\in\Omega$, and $\cvar_\alpha:L^1(\Omega)\to(-\infty,\infty]$ denotes the CVaR risk measure at confidence level $\alpha\in(0,1)$.
For any distribution $\bm{r}$ on $(\Omega, \mathcal{A})$, we write
\[
    \E_{\bm{r}}[F(\mathbold{x},\omega)]
    =\sum_{i=1}^n F(\mathbold{x},\omega_i)r_i
    \eqqcolon \sum_{i=1}^n F_i(\mathbold{x})r_i.
\]

Rockafellar and Uryasev \cite{rockafellar2000optimization} showed that CVaR admits the primal representation
\begin{equation}\label{eq:CVaRPrimal}
    \cvar_\alpha\bigl(F(\mathbold{x}, \omega)\bigr)
    =
    \min_{t \in \mathbb{R}}
    \left\{
    t + \frac{1}{1-\alpha}\E_{\bm{p}}\bigl[(F(\mathbold{x}, \omega) - t)_+\bigr]
    \right\},
\end{equation}
where $(\cdot)_+=\max\{0,\cdot\}$.
This formulation is widely used in software packages \cite{mathworks2024cvar,PSG,riskfolio} because it recasts the original problem as a standard (expectation-based) stochastic optimization problem, adding a single auxiliary scalar variable $t$ to the original decision space $\mathcal{X}$.
However, two drawbacks --- both associated with the partial moment function
$
\E_{\bm{p}}\bigl[(F(\mathbold{x}, \omega) - t)_+\bigr]
$ --- are intrinsic to~\cref{eq:CVaRPrimal}.

First, the partial moment function
is nonsmooth in $(\mathbold{x},t)$ due to $(\cdot)_+$. Thus, from the standpoint of algorithmic optimization theory, standard subgradient-based methods suffer from slow theoretical and practical convergence rates \cite{nesterov2018lectures}. In practice, state-of-the-art variable-metric subgradient methods such as Shor's $r$-algorithm \cite{Shor1998Nondifferentiable} often exhibit substantially faster convergence than standard subgradient schemes on a variety of academic and practical problems. However, the convergence theory for these methods remains incomplete, while extensions to stochastic settings are limited and inefficient.
Classical smoothing ideas \cite{nesterov2018lectures} can significantly improve theoretical convergence rates in nonsmooth optimization, but the resulting smoothing mechanisms are typically nonadaptive and ambiguous, thereby limiting their practical appeal.

Second, when $\bm{p}$ is accessed only via sampling, the associated stochastic subgradients exhibit high variance.
In particular, an increasing fraction of samples satisfy $F(\mathbold{x},\omega)\le t$ as $\alpha$ approaches $1$.
Such samples neither contribute to the objective nor to its (sub)gradient, yet they still incur computational cost.
This limits the efficiency of both sample-average approximation (SAA) and stochastic approximation (SA) schemes in high-confidence (i.e., $\alpha \geq 0.95$) CVaR optimization \cite{robbins1951stochastic,royset2024riskadaptiveapproachesstochasticoptimization}.
Developing a general framework that mitigates \emph{both} the nonsmoothness and sampling inefficiencies of the primal representation~\cref{eq:CVaRPrimal} is the main goal of this paper.

\subsection{Contributions}
Our principal contribution is an algorithmic framework that combines smoothing and adaptive importance sampling via a dual, saddle-point view of CVaR.
Specifically, using the dual representation \cite{rockafellar2006generalized,artzner1999coherent},
\begin{equation}\label{eq:cvardual2}
 \cvar_\alpha\big(F(\mathbold{x},\omega)\big)
    = \max_{\bm{q}\in\mathcal{Q}}~\E_{\bm{q}}\big[F(\mathbold{x},\omega)\big],
\end{equation}
where
\[
\mathcal{Q}
:=
\Big\{
\bm{q}\in \mathcal{Q}_\alpha:
\mathbf{1}^\top \bm{q}=1
\Big\}, \qquad \mathcal{Q}_\alpha:= \Big\{\bm{q}\in \mathbb{R}^n: \bm{0}\le \bm{q}\le \frac{\bm{p}}{1-\alpha}\Big\},
\]
we rewrite~\cref{eq:cvar_opt} as the saddle-point problem
\begin{equation}\label{eq:cvarsaddlepoint}
    \min_{\mathbold{x}\in\mathcal{X}}\ \max_{\bm{q}\in\mathcal{Q}}~\E_{\bm{q}}[F(\mathbold{x},\omega)].
\end{equation}

The maximizer(s) $\bm{q}^\star = \bm{q}^\star(\mathbold{x})$ in~\cref{eq:cvarsaddlepoint}, while usually expensive to compute, characterize the \emph{ideal} tail-focused sampling distribution(s) associated to the decision variable $\mathbold{x}$.
These optimal distributions saturate the upper bound $p_i/(1-\alpha)$ in tail scenarios, with almost all other probabilities equal to zero.\footnote{Fractional mass scenarios can appear at the $\alpha$-quantile, which is also the source of possible non-uniqueness; see, e.g., \cite{RoysetWets2021}.}
This structure is useful for sample estimation but problematic for sampling during optimization.
Indeed, the inherent sparsity of $\bm{q}^\star$ can prevent exploration, leading to prematurely ``locking in'' to a potentially incorrect tail set and making it impossible to accurately update the iterates $\mathbold{x}^k$ thereafter.
A similar conclusion can be drawn if~\cref{eq:cvarsaddlepoint} is treated by quadratic epi-regularization, as proposed in~\cite{kouri2020epi,kouri2022primal}. Therefore, a clear challenge (and objective) is ensuring that every sampling distribution $\bm{q}^k \approx \bm{q}^\star(\mathbold{x}^k)$ belongs to the \emph{interior} of $\mathcal{Q}$.

We design an interior-preserving update rule for $\bm{q}^k$ by applying a (block) Bregman proximal point method to~\cref{eq:cvarsaddlepoint}, with a divergence tailored to the box constraints in $\mathcal{Q}$.
The divergence is generated by a generalized Fermi--Dirac entropy (introduced in \Cref{sec:AlgorithmDerivation}), which acts as a weak barrier-like regularizer for
$\mathcal{Q}_\alpha$, always yielding iterates $\bm{q}^k\in\operatorname{int}\mathcal{Q}$.

\medskip
\noindent\textbf{Main contributions.}
\medskip
\begin{itemize}
    \item \textit{Adaptive smoothing.}
    We introduce a generalized Fermi--Dirac entropy and its associated Bregman divergence to regularize the CVaR envelope $\mathcal{Q}$ and keep $\bm{q}^k$ in its interior. This regularization yields an adaptive CVaR smoothing mechanism within the Bregman proximal point framework.
    \item \textit{Adaptive importance sampling.}
    At each outer iteration, we update $\bm{q}^k$ via a closed-form block Bregman proximal step. This produces an adaptive importance sampler for the inner subproblem that increasingly concentrates on tail scenarios, $\bm{q}^k \to \bm{q}^\star$.
    \item \textit{Theory and practice.}
    We establish convergence guarantees (including variants with inexact inner solves) and demonstrate performance on portfolio optimization, support vector classification, and topology optimization.
\end{itemize}

\medskip
\noindent
A compact version of the proposed Exponential Adaptive Smoothing and Importance Sampling Technique (\textsc{EASIeST}) is given in \Cref{alg:ShortAlg}. A full derivation with precise definitions of the block proximal operator and the smoothed subproblem appears in \Cref{sec:AlgorithmDerivation}.

\begin{algorithm2e}[h]
\DontPrintSemicolon
\caption{\label{alg:ShortAlg} \textsc{EASIeST} for CVaR optimization.}
\SetKwInOut{Input}{Input}
\SetKwInOut{Output}{Output}
\Input{initial $\bm{q}^0 \in \operatorname{int}\mathcal{Q}$, $k=0$.}
\Output{an approximate saddle point $(\mathbold{x}^\star,\bm{q}^\star)\in\mathcal{X}\times\mathcal{Q}$.}
\Repeat{\textsl{converged}}{
\textit{Step 1.} Generate a block $\mathcal{B}_k \subseteq \{1,\ldots,n\}$ with $|\mathcal{B}_k|\ge 2$.\;
\BlankLine
\textit{Step 2.} Compute an (approximate) minimizer $\mathbold{x}_k^\star$ of the smoothed CVaR subproblem associated with $(\bm{q}^k,\gamma_k,\mathcal{B}_k)$ (cf. \eqref{eq:block_epi_cvar}).\;
\BlankLine
\textit{Step 3.} Update $\bm{q}^{k+1} = \operatorname{prox}^{\mathcal{B}_k}_{\gamma_k \varphi}(\bm{q}^k)$ (cf. \cref{eq:Prox_Formula_block}).\;
\BlankLine
$k\leftarrow k+1$.\;
}
\end{algorithm2e}

We briefly comment on Steps~1--2.
The block-generation procedure (Step~1) is described in \Cref{sub:Complete_algorithm}. For now, we note that the blocks are randomly generated by sampling from $\bm{q}^k$. This plays a central role in the method's efficiency and convergence.
Step~2 can be implemented using either the SA or the SAA methods, depending on the application and computational budget. The expression for the epi-smoothed CVaR is derived in \Cref{sub:BregmanForCVaR}. A variant of the SA scheme for computing $\bm{x}_k^\star$ is provided in \Cref{sub:Complete_algorithm}. Practical guidance (including importance sampling using $\bm{q}^k$ and inexactness heuristics) is provided in \Cref{sec:PracticalImplementation}.

\subsection{Paper organization}
\Cref{sec:AlgorithmDerivation} derives the method from a Bregman proximal point view of the saddle-point formulation and introduces the generalized Fermi--Dirac entropy and the resulting block proximal update.
\Cref{sec:Analysis} establishes convergence guarantees.
\Cref{sec:PracticalImplementation} discusses implementation details, including hyperparameter selection, adaptive step-size, and inexactness considerations.
\Cref{sec:NumericalExperiments} reports numerical results in portfolio optimization, support vector classification, and topology optimization.
\Cref{sec:Conclusion} concludes.


\section{Algorithm derivation}
\label{sec:AlgorithmDerivation}

This section derives~\Cref{alg:ShortAlg}. We begin by introducing the essential aspects of the Bregman proximal point method \cite{chen1993convergence, teboulle2018simplified} in the context of our CVaR optimization problem (\Cref{sub:Bregman_divergences,sub:Bregman_P_P}). Next, we introduce the concept of Bregman epi-smoothing of CVaR (\Cref{sub:BregmanEpiReg}).
Relying on the results derived in the aforementioned sections, we rewrite CVaR optimization as a maximization problem over scenario weights $\bm{q}$ on the set $\mathcal{Q}$ and derive a closed-form Bregman proximal step, which reduces each outer update to solving a smoothed CVaR subproblem in the primal variable $\bm{x}$ (\Cref{sub:BregmanForCVaR}).

\subsection{Bregman divergences}\label{sub:Bregman_divergences}

Bregman divergences form a broad class of dissimilarity measures for $n$-dimensional vectors. Introduced by L.M.\ Bregman \cite{BREGMAN1967200}, these divergences generalize the squared Euclidean distance and have since become fundamental tools in a range of scientific and engineering disciplines.
To define a Bregman divergence, we begin with its generating function.
\begin{definition}[Legendre function] Let $\mathcal{P} \subset \mathbb{R}^n$ be closed convex with non\-empty interior. Then a convex function \( \psi: \mathbb{R}^n \to (-\infty, \infty] \) with value $+\infty$ outside $\mathcal{P}$ is Legendre if
\begin{itemize}
    \item[(i)] the restriction of $\psi$ to $\mathcal{P}$ is continuous;
    \item[(ii)] $\psi$ is essentially smooth: continuously differentiable on the interior of its domain, $\operatorname{int} \mathcal{P}$, such that for all $\bm{p} \in \operatorname{bd}\mathcal{P}$ and $\bm{q} \in \operatorname{int}\mathcal{P}$,
    $$\lim_{t \to 0, t>0} \bigl\langle\nabla \psi\bigl(\bm{p} + t(\bm{q}-\bm{p})\bigr),  \bm{q}-\bm{p}\bigr\rangle = -\infty;
    $$
    \item[(iii)] $\psi$ is strictly convex on the interior of its domain, $\operatorname{int}\mathcal{P}$.
\end{itemize}  
\end{definition}

For any such Legendre function $\psi$, the associated Bregman divergence \( D_\psi(\mathbold{p}, \mathbold{q}) \) between two points $\mathbold{p} \in \mathcal{P}$ and $\mathbold{q} \in \operatorname{int} \mathcal{P}$ is defined as:
\begin{equation}\label{eq:bregmandivgen}
D_\psi(\mathbold{p}, \mathbold{q}) \coloneqq \psi(\mathbold{p}) - \psi(\mathbold{q}) - \langle \nabla \psi(\mathbold{q}), \mathbold{p} - \mathbold{q} \rangle
\,,
\end{equation}
where \( \langle \cdot, \cdot \rangle \) represents the inner product in \( \mathbb{R}^n \), and $D_\psi(\mathbold{p}, \mathbold{q}) \coloneqq +\infty$ for all other $\bm{p},\bm{q} \in \R^n$.

One of the significant features of this divergence is its separable form,
\begin{equation}\label{eq:bregmandivsep}
    D_\psi(\bm{p}, \bm{q}) = \sum_{i=1}^n \left[\psi_i(p_i) - \psi_i(q_i) -\psi_i'(q_i)(p_i-q_i)\right]
    \,,
\end{equation}
that appears when $\psi(\bm{q}) \coloneqq \sum_{i=1}^n \psi_i(q_i)$ is defined as a sum of component  functions $\psi_i \colon \mathbb{R} \to \mathbb{R}$.
For instance, the functions $\psi_i(x) = \frac{1}{2}(x-1)^2$ yield the one half the standard squared Euclidean distance
\begin{equation}
\label{eq:MoreauDivergence}
D_\psi(\bm{p}, \bm{q}) = \frac{1}{2}\sum_{i=1}^n (p_i-q_i)^2
\,.
\end{equation}

Bregman distances have been extensively utilized and studied in optimization theory \cite{nemirovski1983interior, chen1993convergence, csiszar2008minimization, csiszar2009minimization, laude2026allroadsrome} in their general form~\cref{eq:bregmandivgen}. On the other hand, in information theory and statistics, they are typically considered in the separable form~\cref{eq:bregmandivsep}, where vectors $\bm{p}$ and $\bm{q}$ represent generalized distributions (finite discrete measures) \cite{csiszar1991least, csiszar1994maximum,csiszar1995generalized}.
As demonstrated below, we use the Bregman divergence to construct a distance function on a subset of a probability simplex, wherein the separable form is convenient for us.
The next section utilizes a Bregman divergence within the framework of the proximal point method.

\subsection{Bregman proximal point}\label{sub:Bregman_P_P}

Bregman proximal point generalizes the classical proximal point method of Martinet \cite{martinet1970regularisation} and Rockafellar \cite{rockafellar1976monotone}. For a convex objective $\varphi:\mathcal{Q}\to(-\infty,\infty]$, it generates a sequence $\{\bm{q}^k\}_{k\in\mathbb{N}}$ in $\mathcal{Q}$ via
\begin{equation}\label{eq:ProxPointAlg}
\bm{q}^0\in\operatorname{int}\operatorname{dom}\psi,\qquad
\bm{q}^{k+1}=\prox_{\gamma_k\varphi}(\bm{q}^k),\qquad k=0,1,\ldots,
\end{equation}
where $\{\gamma_k\}\subset(0,\infty)$ satisfies $\sum_{k=0}^\infty \gamma_k=+\infty$, and the (Bregman) proximal operator is
\begin{equation}\label{eq:ProxOperator}
\prox_{\gamma\varphi}(\bm{r})
\coloneqq
\argmin_{\bm{q}\in\mathcal{Q}}
\Big\{\gamma \varphi(\bm{q})+D_\psi(\bm{q},\bm{r})\Big\}.
\end{equation}
When $\psi(\bm{q})=\tfrac12\|\bm{q}\|_2^2$ (cf.~\eqref{eq:MoreauDivergence}), one recovers the classical proximal point method.

\paragraph{Property 1 (function-value convergence)}
The following standard estimate will be used repeatedly.

\begin{theorem}[Theorem 3.4 in \cite{chen1993convergence}]
\label{thm:BregmanConvergence}
Assume $\varphi:\R^n\to(-\infty,\infty]$ is proper, lower semicontinuous, and convex. Then the iterates~\eqref{eq:ProxPointAlg} satisfy
\begin{equation}\label{eq:BregmanConvergence}
\varphi(\bm{q}^k)-\varphi(\bm{q}^\star)
\le
\frac{D_\psi(\bm{q}^\star,\bm{q}^0)}{\sum_{i=0}^{k-1}\gamma_i},
\end{equation}
where $\bm{q}^\star\in\argmin\limits_{\bm{q}\in\mathcal{Q}}\varphi(\bm{q})$.
\end{theorem}

\noindent\textit{Proof sketch.}
The result follows from (i) optimality of $\bm{q}^{k+1}$ in~\eqref{eq:ProxOperator}, (ii) the three-point identity for $D_\psi$, and (iii) the subgradient inequality for $\varphi$, yielding a telescoping bound on $D_\psi(\bm{q}^\star,\bm{q}^k)$ and hence~\eqref{eq:BregmanConvergence}.
The full proof is included in~\Cref{app:thm_bregman_convergence} for completeness.

\begin{definition}[Linear convergence rates]
We say that a sequence $\{\bm{x}^k\}_{k\in\mathbb{N}}$ converges to $\bm{x}^\star$ with order $q\ge 1$ and rate $r\ge 0$ if
\[
\lim_{k\to\infty}\frac{\|\bm{x}^{k+1}-\bm{x}^\star\|}{\|\bm{x}^k-\bm{x}^\star\|^q}=r.
\]
If $q=1$ and $r\in(0,1)$, then $\bm{x}^k$ converges \emph{$\mathrm{Q}$-linearly} to $\bm{x}^\star$.
If $\|\bm{x}^k-\bm{x}^\star\|\le \epsilon_k$ for all $k$, where $\epsilon_k$ converges $\mathrm{Q}$-linearly to $0$, then $\bm{x}^k$ converges \emph{$\mathrm{R}$-linearly} to $\bm{x}^\star$.
\end{definition}

\begin{corollary}
\label{cor:geomsteps}
If the step sizes grow geometrically, i.e., $\gamma_{k+1}=c\gamma^{k}$ for some $c>0$ and $\gamma>1$, then the values $\varphi(\bm{q}^k)$ from~\eqref{eq:ProxPointAlg} converge at least $\mathrm{R}$-linearly to $\varphi(\bm{q}^\star)$ with rate $1/\gamma$.
\end{corollary}

\begin{proof}
Define $\epsilon_k\coloneqq D_\psi(\bm{q}^\star,\bm{q}^0)/\sum_{i=0}^{k-1}\gamma_i$ and recall
$\sum_{j=0}^{k-1}\gamma^j=(1-\gamma^k)/(1-\gamma)$.
Then for $\gamma_{k+1}=c\gamma^k$,
\[
\frac{\epsilon_{k+1}}{\epsilon_k}
=
\frac{1-\gamma^k}{1-\gamma^{k+1}}
\to
\frac{1}{\gamma}
\qquad\text{as }k\to\infty.
\]
\end{proof}

\paragraph{Property 2 (interior iterates)}
If $\operatorname{dom}\psi=\mathcal{Q}$, then each iterate $\bm{q}^k$ remains in $\operatorname{int}\mathcal{Q}$.

\paragraph{Property 3 (convergence in iterates)}
One more key feature of the Bregman proximal point method is that the Bregman divergence generated by the Legendre function $\psi$ acts as a Lyapunov function along the iterates. In particular, if $\bm{q}^\star$ is any minimizer of the target convex objective and $\bm{q}^{k}$ is obtained from the Bregman proximal update, then the sequence $\{D_\psi(\bm{q}^\star,\bm{q}^k)\}_{k\ge 0}$ is nonincreasing and strictly decreases whenever $\bm{q}^k\neq \bm{q}^\star$. This Fej\'er-type monotonicity implies that $\{\bm{q}^k\}$ is bounded and asymptotically regular (i.e., $D_\psi(\bm{q}^{k+1},\bm{q}^k)\to 0$). Moreover, under the stronger assumption that $\psi$ is strongly convex
with respect to the norm $\|\cdot\|$ used to define the geometry of the method,
the corresponding Bregman divergence controls distances: $D_\psi(\bm{q},\bm{r})\ge c\,\|\bm{q}-\bm{r}\|^2$ for some $c>0$ on the relevant domain. Consequently, the Lyapunov decrease of $D_\psi(\bm{q}^\star,\bm{q}^k)$ yields $\|\bm{q}^k-\bm{q}^\star\|\to 0$, i.e., convergence of iterates in the norm with respect to which $\psi$ is strictly convex.
\begin{remark}[The choice of Legendre function]\label{rem:natural-psi-simplex}
When the iterates are probability distributions (e.g., $\bm q^k\in\mathcal{Q}$), the choice of Legendre function $\psi$ is not merely aesthetic: it determines the \emph{notion of proximity} enforced by the proximal step and, consequently, the qualitative behavior of the iterates. A Euclidean choice such as $\psi(\bm q)=\tfrac12\|\bm q\|_2^2$ induces the $\ell_2$ geometry, which can severely under-penalize redistributions of mass in high dimensions. Indeed, for
\[
\bm q=\Big(\tfrac1n,\ldots,\tfrac1n\Big),\qquad 
\bm r=\Big(\underbrace{0,\ldots,0}_{\alpha n},\tfrac{1}{(1-\alpha)n},\ldots,\tfrac{1}{(1-\alpha)n}\Big),
\]
one has $\tfrac12\|\bm q-\bm r\|_1=\alpha$, i.e., $\bm r$ differs from $\bm q$ by moving an $\alpha$-fraction of probability mass, which is a \emph{large} perturbation in a distributional sense. However,
\[
\|\bm q-\bm r\|_2=\sqrt{\frac{\alpha}{(1-\alpha)n}},
\]
which becomes small as $n$ grows (e.g., $\|\bm q-\bm r\|_2\approx 3\times 10^{-3}$ for $\alpha=0.9$ and $n=10^6$). Thus, an $\ell_2$-based proximal term may treat dramatically different distributions as “close,” allowing large mass transfers at negligible Euclidean cost. In contrast, entropy-type Legendre functions (e.g., the generalized Fermi--Dirac entropy) generate Bregman divergences that are aligned with the simplex geometry and better reflect statistically meaningful discrepancies (e.g., via Pinsker-type controls \cite{pinsker1964,tsybakov2009}). This is particularly important when the iterate $\bm q^k$ is employed as a sampling distribution (as done here): the natural geometry helps keep $\bm q^k$ in the interior and prevents premature degeneracy of the sampler.
\end{remark}

A potential drawback is that $\prox_{\gamma\varphi}$ is not always efficiently computable. However, for our CVaR dual set $\mathcal{Q}_\alpha$ and an appropriate choice of $\psi$, we can derive a convenient proximal operator in closed-form.

\subsection{Bregman epi-regularization of CVaR}
\label{sub:BregmanEpiReg}

Epi-smoothing (or epi-regularization) of extended real-valued convex functions (or functionals) via infimal convolution (with sufficiently smooth kernels) traces its origins to the foundational work of Moreau \cite{moreau1965proximite}. By choosing the squared Euclidean norm as a smoothing kernel, one recovers the Moreau envelope, also known as Moreau--Yosida regularization, a well-established tool in optimization and convex analysis. In optimization, Moreau--Yosida regularization gives rise to the proximal point algorithm, which plays a crucial role in our work.

Recently, Kouri and Surowiec \cite{kouri2020epi} applied epi-smoothing to convex risk measures and investigated the properties of the resulting measures. In turn, we apply the same technique to smooth CVaR with an appropriately chosen Bregman divergence. Such a divergence (as a smoothing kernel) allows us to \emph{(i)} smooth the CVaR adaptively; and \emph{(ii)} design an importance-sampling procedure, which is equally significant for optimization.
\smallskip

\paragraph{Epi-regularization in dual form}
A coherent risk measure $\mathcal{R}$ admits a dual representation as a supremum of expectations over a convex set of densities/likelihood ratios (see, e.g., \cite{quadrangle}). Epi-regularization adds a strongly convex penalty to this dual formulation. In our finite-sample setting, $X$ is represented by a vector $(X_1,\ldots,X_n)^\top$ and expectations reduce to weighted sums.

Concretely, let $\mathcal{P}\subset\R^n$ be a convex dual feasible set (for CVaR, $\mathcal{P}=\mathcal{Q}$). For a proper closed convex penalty $\psi$ and $\gamma>0$, define the dual-smoothed functional
\begin{equation}\label{eq:epiriskcoherentdual_minimal}
\widetilde{\mathcal{R}}^\gamma(X)
=
\sup_{\bm{q}\in\mathcal{P}}
\left\{
\E_{\bm{q}}[X]-\frac{1}{\gamma}\psi(\bm{q})
\right\}.
\end{equation}
For CVaR, we will apply this construction first to the associated \emph{coherent regret} \cite{quadrangle} (which has a simpler dual set) and then recover the risk via the standard regret-to-risk formula \cite{quadrangle}; we highlight this switch explicitly below.

\begin{example}[Exponential smoothing of CVaR]\label{ex:ExpSmoothCVaR}
Let $\mathcal{R}_\alpha(X)=\CVaR_\alpha(X)$. 

Then, cf.~\eqref{eq:cvardual2},
\begin{equation}\label{eq:cvardual3}
\mathcal{R}_\alpha(X)=\max_{\bm{q}\in\mathcal{Q}}\ \E_{\bm{q}}[X].
\end{equation}

\paragraph{Step 1}
The coherent regret associated with CVaR is
\[
\mathcal{V}_\alpha(X)=\frac{1}{1-\alpha}\E_{\bm{p}}[X_+],
\]
and it admits the dual representation
\begin{equation}\label{eq:cvarregretdual}
\mathcal{V}_\alpha(X)
=
\max_{\bm{q}\in\mathcal{Q}_\alpha}\ \langle \bm{q},X\rangle.
\end{equation}
Introduce the generalized Fermi--Dirac binary entropy
\begin{equation}\label{eq:FermiDirac}
\psi(\bm{q})
\coloneqq
\sum_{i=1}^n
\Big[
q_i\ln q_i+\Big(\frac{p_i}{1-\alpha}-q_i\Big)\ln\Big(\frac{p_i}{1-\alpha}-q_i\Big)
\Big],
\end{equation}
for $\bm{q}\in\mathcal{Q}_\alpha$, and set $\psi(\bm{q})=+\infty$ otherwise.
\begin{lemma}[Generalized Pinsker's inequality]\label{lem:gen-pinsker}
Let $\psi$ be the generalized Fermi--Dirac entropy and $D_{\mathrm{KL}}$ denote the Kullback--Leibler divergence. Then for any probability distributions $\bm{q}, \bm{r} \in \operatorname{int} \mathcal{Q}$, the following chain of inequalities holds:

$$
D_\psi(\mathbold{q},\mathbold{r})
\;\ge\; D_{\mathrm{KL}}(\mathbold{q},\mathbold{r})
\;\ge\; \frac12\|\mathbold{q}-\mathbold{r}\|_1^2 \, ,$$
and in particular
$$\psi(\bm{q}) \geq \psi(\bm{r}) + \langle  \nabla \psi(\bm{r}), \bm{q} -\bm{r} \rangle + \frac{1}{2}\|\mathbold{q}-\mathbold{r}\|_1^2 \, .
$$
\end{lemma}

\begin{proof}
For $\bm{q}, \bm{r} \in \operatorname{int} \mathcal{Q}$, define $D_{\mathrm{KL}}(\mathbold{q},\mathbold{r}): = \sum\limits_{i=1}^n q_i \ln \dfrac{q_i}{r_i}$, then (cf. \eqref{eq:Bregman_div_formula}) 
\begin{equation}\label{eq:pinsker_chain}
    \begin{split}
     D_\psi(\mathbold{q},\mathbold{r}) &= D_{\mathrm{KL}}(\mathbold{q},\mathbold{r}) + \sum_{i=1}^n\left(\dfrac{p_i}{1-\alpha} -q_i\right)\ln \dfrac{p_i/(1-\alpha) - q_i}{p_i/(1-\alpha) - r_i}\\
     &\geq D_{\mathrm{KL}}(\mathbold{q},\mathbold{r}) \, .
    \end{split}
\end{equation}
Therefore, applying Pinsker's inequality to \eqref{eq:pinsker_chain} completes the proof. 
\end{proof}

For $\gamma>0$, define the epi-regularized regret
\begin{equation}\label{eq:epiregret}
\widetilde{\mathcal{V}}^\gamma_\alpha(X)
=
\max_{\bm{q}\in\R^n}
\left\{
\langle \bm{q},X\rangle-\frac{1}{\gamma}\psi(\bm{q})
\right\}.
\end{equation}
Since $\psi(\bm{q})=+\infty$ outside $\mathcal{Q}_\alpha$,~\eqref{eq:epiregret} is an unconstrained concave maximization problem and the maximizer is unique.
A straightforward calculation yields
\begin{align}
\tilde{\bm{q}}^{\gamma}(X)
&=
\nabla\psi^{-1}(\gamma X)
=
\frac{\bm{p}}{1-\alpha}\,
\frac{\exp(\gamma X)}{1+\exp(\gamma X)}
\qquad\text{(component-wise)}, \\
\widetilde{\mathcal{V}}^\gamma_\alpha(X)
&=
\frac{1}{\gamma(1-\alpha)}\E_{\bm{p}}\ln\big(1+\exp(\gamma X)\big)
-
\underbrace{\frac{1}{\gamma(1-\alpha)}\E_{\bm{p}}\ln\frac{\bm{p}}{1-\alpha}}_{\text{constant independent of }X}.
\end{align}

\paragraph{Step 2}
Therefore, the exponentially smoothed CVaR is
\begin{equation}\label{eq:expsmoothedcvar}
\widetilde{\cvar}^{\gamma}_{\alpha}(X)
=
\min_{t\in\R}\left\{
t+\frac{1}{\gamma(1-\alpha)}\E_{\bm{p}}\ln\big(1+\exp(\gamma(X-t))\big)
 + C_{\alpha}^{\gamma,\bm{p}}
\right\} \, ,
\end{equation}
where $C_{\alpha}^{\gamma,\bm{p}}=-\dfrac{1}{\gamma(1-\alpha)}\E_{\bm{p}}\ln\dfrac{\bm{p}}{1-\alpha}$.
\end{example}

\begin{remark}[Adjusted exponentially smoothed CVaR]
To remove the constant term $C_{\alpha}^{\gamma,\bm{p}}$ in~\eqref{eq:expsmoothedcvar}, define
\[
\hat{\psi}(\bm{q})
=
\psi(\bm{q})+\sum_{i=1}^n \frac{p_i}{1-\alpha}\ln\frac{p_i}{1-\alpha},
\]
which yields
\begin{equation}\label{eq:expsmoothedcvarnorm}
\widetilde{\cvar}_{\alpha}^{\gamma}(X)
=
\min_{t\in\R}
\left\{
t+\frac{1}{\gamma(1-\alpha)}\E_{\bm{p}}\ln\big(1+\exp(\gamma(X-t))\big)
\right\}.
\end{equation}
This type of smoothing is discussed in \cite{royset2024riskadaptiveapproachesstochasticoptimization} from the primal perspective.
\end{remark}

In this paper, we use an alternative smoothing that is \emph{relative to a reference distribution} $\bm{r}\in\operatorname{int}\mathcal{Q}$. Specifically, for such $\bm{r}$ we use the Bregman divergence generated by~\eqref{eq:FermiDirac},
\begin{equation}\label{eq:Bregman_div_formula}
D_\psi(\bm{q},\bm{r})
=
\sum_{i=1}^n
q_i\ln\frac{q_i}{r_i}
+
\Big(\frac{p_i}{1-\alpha}-q_i\Big)\ln\frac{\frac{p_i}{1-\alpha}-q_i}{\frac{p_i}{1-\alpha}-r_i},
\end{equation}
for $\bm{q}\in\mathcal{Q}$, and $D_\psi(\bm{q},\bm{r})=+\infty$ otherwise.

We then define the epi-smoothed risk via the \emph{Bregman kernel} $D_\psi(\cdot,\bm{r})$:
\begin{equation}\label{eq:Bregman CVaR formula}
\widetilde{\mathcal{R}}^{\gamma, \bm{r}}(X)
=
\min_{t\in\R}
\Big\{
t+\widetilde{\mathcal{V}}^{\gamma, \bm{r}}(X-t)
\Big\},
\end{equation}
where
\begin{equation}\label{eq:Bregman regret}
\widetilde{\mathcal{V}}^{\gamma, \bm{r}}(X)
:=
\max_{\bm{q}\in\R^n}
\left\{
\langle \bm{q},X\rangle-\frac{1}{\gamma}D_\psi(\bm{q},\bm{r})
\right\}.
\end{equation}
\begin{proposition}[Bregman epi-smoothed CVaR]
\label{prop:Bregman epi-regularized CVaR}
Let $D_\psi$ be the Bregman divergence associated with the Fermi--Dirac entropy~\eqref{eq:FermiDirac}. Then for any $\bm{r}\in\operatorname{int}\mathcal{Q}$,
\begin{align}
\tilde{\bm{q}}^{\gamma, \bm{r}}(X)
&=
\nabla\psi^{-1}\big(\nabla\psi(\bm{r})+\gamma X\big)
\in\operatorname{int}\mathcal{Q}_\alpha,
\label{eq:Bregman q}\\
\widetilde{\mathcal{V}}_\alpha^{\gamma,\bm{r}}(X)
&=
\frac{1}{\gamma(1-\alpha)}\E_{\bm{p}}
\ln\Big(1+\exp\big(\nabla\psi(\bm{r})+\gamma X\big)\Big)
+
C_{\alpha}^{\bm{p},\bm{r}, \gamma},
\label{eq:Bregman regret explicit}\\
\widetilde{\cvar}_{\alpha}^{\gamma, \bm{r}}(X)
&=
\min_{t\in\R}
\left\{
t+\frac{1}{\gamma(1-\alpha)}\E_{\bm{p}}
\ln\Big(1+\exp\big(\nabla\psi(\bm{r})+\gamma(X-t)\big)\Big)
+
C_{\alpha}^{\bm{p},\bm{r}, \gamma}
\right\},
\label{eq:Bregman CVaR}
\end{align}
where
\[
C_{\alpha}^{\bm{p},\bm{r}, \gamma}
:=
\frac{1}{\gamma(1-\alpha)}\E_{\bm{p}}
\ln\frac{\bm{p}/(1-\alpha)-\bm{r}}{\bm{p}/(1-\alpha)}.
\]
\end{proposition}

\begin{proof}
Applying first-order optimality to~\eqref{eq:Bregman regret} yields
\[
X-\frac{1}{\gamma}\big(\nabla\psi(\bm{q})-\nabla\psi(\bm{r})\big)=\bm{0}
\quad\Longleftrightarrow\quad
\bm{q}=\nabla\psi^{-1}\big(\nabla\psi(\bm{r})+\gamma X\big),
\]
which proves~\eqref{eq:Bregman q}. Substituting the optimizer into~\eqref{eq:Bregman regret} gives~\eqref{eq:Bregman regret explicit}. Finally, using~\eqref{eq:Bregman CVaR formula} yields~\eqref{eq:Bregman CVaR}.
\end{proof}

\begin{corollary}[Optimal $t^\star$ in~\eqref{eq:Bregman CVaR}]
\label{cor:optimal_t}
For the Bregman epi-smoothed CVaR \eqref{eq:Bregman CVaR}, the minimizer $t^\star$ satisfies
\begin{equation}\label{eq:optimal t for CVaR}
\mathbf{1}^\top \tilde{\bm{q}}^{\gamma, \bm{r}}\big(X-t^\star\big)=1 \quad \text{and} \quad \tilde{\bm{q}}^{\gamma, \bm{r}}\big(X-t^\star\big) \in \operatorname{int} \mathcal{Q}.
\end{equation}
\end{corollary}
The next \Cref{sub:BregmanForCVaR} provides the technical details and relevant aspects of this approach applied to \eqref{eq:cvar_opt}. Here, we mention only that the Bregman smoothing method \eqref{eq:Bregman CVaR} is relative to a chosen distribution $\bm{r} \in \operatorname{int} \mathcal{Q}$, which makes the smoothing adaptive in the Bregman proximal point method framework---something that classical exponential smoothing \eqref{eq:expsmoothedcvar} lacks.

\subsection{Bregman proximal point for CVaR optimization}
\label{sub:BregmanForCVaR}

Define
\begin{equation}\label{eq:phi_for_sub}
\varphi(\bm{q})
\coloneqq
\begin{cases}
\displaystyle
-\min_{\bm{x}\in\mathcal{X}}\ \E_{\bm{q}}\big[F(\bm{x},\omega)\big],
&\text{if }\mathbf{1}^\top\bm{q}=1,\\[6pt]
+\infty,&\text{otherwise,}
\end{cases}    
\end{equation}
and rewrite~\cref{eq:cvarsaddlepoint} as the convex minimization problem
\begin{equation}\label{eq:DualProblem}
\min_{\bm{q}\in\mathcal{Q}}\ \varphi(\bm{q}) \, .
\end{equation}
Here, $\varphi$ is convex as the negative pointwise minimum of linear functions in $\bm{q}$.
We apply the Bregman proximal point method with the Fermi--Dirac $\psi$ to~\eqref{eq:DualProblem}. The next result gives the key closed-form proximal step and smoothed CVaR primal subproblem structure behind~\Cref{alg:ShortAlg}.

For $\bm r\in\operatorname{int}\mathcal Q$, $\gamma>0$,
$\bm x\in\mathcal X$, and $t\in\mathbb R$, define
\begin{equation}\label{eq:q_update_map}
\bm q_{\gamma,\bm r}(\bm x,t)
:=
\nabla\psi^{-1}\!\Big(
\nabla\psi(\bm r)
+
\gamma\big(F(\bm x,\omega)-t\big)
\Big).
\end{equation}

\begin{proposition}[Bregman step for CVaR]
\label{lem:Prox}
Assume $F:\mathcal{X}\times\Omega\to\R$ is convex in $\bm{x}\in\mathcal{X}$ for all $\omega\in\Omega$. Then for all $\bm{r}\in\operatorname{int}\mathcal{Q}$,
\begin{subequations}
\begin{equation}\label{eq:Prox_Formula}
\prox_{\gamma\varphi}(\bm{r})
= \bm q_{\gamma,\bm r}(\overline{\bm{x}},\overline{t})
\in\operatorname{int}\mathcal{Q},
\end{equation}
where $\overline{\bm{x}}$ solves the regularized subproblem
\begin{equation}\label{eq:Prox_RegularizedSubproblem}
\overline{\bm{x}}
\in
\argmin_{\bm{x}\in\mathcal{X}}
\ \widetilde{\cvar}_{\alpha}^{\gamma, \bm{r}}\big(F(\bm{x},\omega)\big),
\end{equation}
and $\overline{t}$ is chosen so that $\mathbf{1}^\top\bm q_{\gamma,\bm r}(\overline{\bm{x}},\overline{t}) =1$.
\end{subequations}
\end{proposition}

\begin{proof}
Using strong duality, we can write the proximal step as
\begin{equation}\label{eq:Prox_StrongDuality}
\min_{\bm{x}\in\mathcal{X},\,t\in\R}\ \max_{\bm{q}\in\R^n}\ \frac{1}{\gamma}\mathcal{L}_{\bm{r}}(\bm{q},\bm{x},t),
\end{equation}
where (setting a Lagrange multiplier $\lambda=\gamma t$ for the constraint $\mathbf{1}^\top\bm{q}=1$)
\[
\mathcal{L}_{\bm{r}}(\bm{q},\bm{x},t)
=
\gamma t+\gamma\langle \bm{q},F(\bm{x},\omega)-t\rangle-D_\psi(\bm{q},\bm{r}).
\]
Applying \Cref{prop:Bregman epi-regularized CVaR} to the inner maximization over $\bm{q}$ yields the optimizer
$\bm{q}=\nabla\psi^{-1}\big(\nabla\psi(\bm{r})+\gamma(F(\bm{x},\omega)-t)\big)$, and the resulting reduced objective in $\bm{x}$ is precisely
$\widetilde{\cvar}_{\alpha}^{\gamma, \bm{r}}(F(\bm{x},\omega))$ up to constants (with implicit minimization in $t$). Minimizing over $\bm{x}$ gives~\eqref{eq:Prox_RegularizedSubproblem}, and substituting the minimizers and applying \Cref{cor:optimal_t} yields~\eqref{eq:Prox_Formula}.
\end{proof}

\begin{remark}[Notation]
For each $i=1,\ldots,n$, define
\begin{equation}\label{eq:Logexp}
h_i(\bm{x},t;\bm{r})
\coloneqq
\ln\Big(1+\exp\big(\nabla_i\psi(r_i)+\gamma(F_i(\bm{x})-t)\big)\Big),
\end{equation}
and the vector form
\begin{equation}\label{eq:LogexpVec}
\bm{h}(\bm{x},t;\bm{r})
\coloneqq
\ln\Big(1+\exp\big(\nabla\psi(\bm{r})+\gamma(F(\bm{x})-t)\big)\Big),
\end{equation}
where $\nabla_i\psi(r_i)$ denotes the scalar derivative of $\psi_i$ at $r_i$, and all exponentials/logarithms are component-wise.
We also denote
\begin{equation}\label{eq:prox_notation}
\overline{\bm q}
:=
\prox_{\gamma\varphi}(\bm r)
=
\bm q_{\gamma,\bm r}(\overline{\bm x},\overline t).
\end{equation}
\end{remark}

The full Bregman step in \Cref{lem:Prox} updates all components of the dual
probability vector $\bm q$. While this is natural in the deterministic setting,
it may be unnecessarily expensive when the number of scenarios is large or when
the objective is accessed through stochastic estimates. We therefore introduce a
block version of the Bregman step, in which only the components indexed by a
subset $\mathcal B\subseteq\{1,\ldots,n\}$ are updated, while the remaining
components are kept fixed. This preserves the structure of the full Bregman
update on the active block, but reduces the amount of information required at
each iteration.

The block formulation is particularly useful in the stochastic setting, where only stochastic estimates of the relevant quantities (e.g., subgradient and objective) are available. The
current dual probability vector can be used to select a block of scenarios and
to define the corresponding importance-sampling weights. Thus, the method
updates the dual distribution only on the sampled scenarios, while retaining the
previous probabilities on the complement. Since the Bregman geometry keeps the
dual probabilities in the interior of the risk envelope, all scenarios remain
eligible for future sampling, which provides a natural exploration mechanism.

For $\bm r\in\operatorname{int}\mathcal Q$, $\gamma>0$,
$\bm x\in\mathcal X$, $t\in\mathbb R$, and a block
$\mathcal B\subseteq\{1,\ldots,n\}$, define the block dual update vector
$\bm q_{\gamma,\bm r}^{\mathcal B}(\bm x,t)$ componentwise by
\begin{equation}\label{eq:q_update_map_block}
\big[\bm q_{\gamma,\bm r}^{\mathcal B}(\bm x,t)\big]_i
:=
\begin{cases}
\nabla_i\psi^{-1}\!\Big(
\nabla_i\psi(r_i)+\gamma\big(F_i(\bm x)-t\big)
\Big),
& i\in\mathcal B,
\\[0.4em]
r_i,
& i\in\mathcal B^c.
\end{cases}
\end{equation}

\begin{corollary}[Block Bregman step for CVaR]
\label{cor:BlockBregmanStep}
Let the assumptions of \Cref{lem:Prox} hold. Fix a block
$\mathcal B\subseteq\{1,\ldots,n\}$ with $|\mathcal B|\ge 2$ and let
$\mathcal B^c:=\{1,\ldots,n\}\setminus\mathcal B$. Then, for every
$\bm r\in\operatorname{int}\mathcal Q$ and $\gamma>0$,
\begin{equation}\label{eq:Prox_Formula_block}
\prox_{\gamma\varphi}^{\mathcal B}(\bm r)
=
\bm q_{\gamma,\bm r}^{\mathcal B}(\hat{\bm x},\hat t)
\in\operatorname{int}\mathcal Q,
\end{equation}
where $(\hat{\bm x},\hat t)$ solves
\begin{equation}\label{eq:Prox_RegularizedSubproblem_block}
(\hat{\bm x},\hat t)
\in
\argmin_{\bm x\in\mathcal X,\,t\in\mathbb R}
\left\{
t
+
\sum_{i\in\mathcal B^c} r_i\big(F_i(\bm x)-t\big)
+
\frac{1}{\gamma(1-\alpha)}
\sum_{i\in\mathcal B} p_i\,h_i(\bm x,t;\bm r)
\right\}.
\end{equation}
\end{corollary}

\begin{proof}
Repeat the proof of \Cref{lem:Prox}, treating the coordinates $q_i=r_i$ for $i\in\mathcal{B}^c$ as fixed and optimizing only over $\{q_i:i\in\mathcal{B}\}$.
\end{proof}

\begin{remark}[Block notation]
Equation~\eqref{eq:Prox_Formula_block} defines the block Bregman proximal operator, denoted $\prox_{\gamma\varphi}^{\mathcal{B}}$: it coincides with $\prox_{\gamma\varphi}$ on $\mathcal{B}$ and is the identity on $\mathcal{B}^c$. Following~\eqref{eq:prox_notation}, we write
\[
\overline{\bm q}^{\mathcal{B}}
:=
\prox_{\gamma\varphi}(\bm r) = \bm{q}^{\mathcal{B}}_{\gamma,\bm r}(\hat{\bm{x}},\hat{t}).
\]
Define the block logistic functions
\begin{align}
h_i^{\mathcal{B}}(\bm{x},t;\bm{r})
&:=
\frac{1}{\gamma(1-\alpha)}\frac{p_i}{r_i}\,h_i(\bm{x},t;\bm{r}),
\qquad i\in\mathcal{B},\\
h_i^{\mathcal{B}}(\bm{x},t;\bm{r})
&:=
F_i(\bm{x})-t,
\hspace{2.6cm} i\in\mathcal{B}^c,
\end{align}
and let $\bm{h}^{\mathcal{B}}(\bm{x},t;\bm{r})$ be the random variable taking values $h_i^{\mathcal{B}}(\bm{x},t;\bm{r})$ with probabilities $r_i$. Then the objective in~\eqref{eq:Prox_RegularizedSubproblem_block} can be written as
\begin{equation}\label{eq:ProxSubObj_block}
f^{\mathcal{B}}(\bm{x},t;\bm{r})
:=
t+\E_{\bm{r}}\big[\bm{h}^{\mathcal{B}}(\bm{x},t;\bm{r})\big].
\end{equation}
Finally, define the block-Bregman epi-smoothed CVaR as
\begin{equation}\label{eq:block_epi_cvar}
\widetilde{\cvar}_{\alpha, \mathcal{B}}^{\gamma, \bm{r}}\big(F(\bm{x},\omega)\big)
:=
\min_{t\in\R}\ f^{\mathcal{B}}(\bm{x},t;\bm{r}).
\end{equation}
\end{remark}

\subsection{Optimality conditions}
\label{sub:OptimalityCond}

This section derives optimality conditions for minimizing the block-Bregman epi-smoothed CVaR~\eqref{eq:block_epi_cvar}, i.e.,
\begin{equation}\label{eq:block_cvar_opt}
\min_{\bm{x}\in\mathcal{X}}\ \widetilde{\cvar}_{\alpha, \mathcal{B}}^{\gamma, \bm{r}}\big(F(\bm{x},\omega)\big).
\end{equation}
We first treat the full-block case $\mathcal{B}=\{1,\ldots,n\}$.

\begin{proposition}[Optimality conditions]
\label{thm:OptimalityConditions}
Under the assumptions of \Cref{lem:Prox}, any optimal decision variable $\overline{\bm{x}}$ and the corresponding optimal multiplier $\overline{t}$ satisfy
\begin{subequations}\label{eq:OptimalityConditions}
\begin{gather}
\label{eq:OptimalityConditions_Feasibility}
\mathbf{1}^\top\overline{\bm q}=1,\\[4pt]
\label{eq:OptimalityConditions_Gradient}
-\E_{\bm{q}^+(\overline{\bm{x}},\overline{t})}\big[g_F(\overline{\bm{x}},\omega)\big]\in \mathcal{N}_{\mathcal{X}}(\overline{\bm{x}}),
\end{gather}
\end{subequations}
where $\mathcal{N}_{\mathcal{X}}(\overline{\bm{x}})$ is the normal cone to $\mathcal{X}$ at $\overline{\bm{x}}$ and $g_F(\overline{\bm{x}},\omega)\in\partial F(\overline{\bm{x}},\omega)$ is a (measurable) subgradient of $F$ in $\bm{x}$.
\end{proposition}

\begin{proof}
Equation~\eqref{eq:OptimalityConditions_Feasibility} restates the defining condition for $\overline{t}$.
Given $\overline{t}$, the inclusion~\eqref{eq:OptimalityConditions_Gradient} follows from first-order optimality for the convex problem~\eqref{eq:Prox_RegularizedSubproblem}.
\end{proof}

\begin{corollary}[Block optimality conditions]
\label{cor:BlockOpt}
Under the assumptions of \allowbreak\Cref{cor:BlockBregmanStep}, any optimal $\hat{\bm{x}}$ and corresponding $\hat{t}$ satisfy
\begin{subequations}\label{eq:OptimalityConditions_block}
\begin{gather}
\label{eq:OptimalityConditions_Feasibility_block}
\mathbf{1}^\top \overline{\bm q}^{\mathcal{B}}=1,\\[4pt]
\label{eq:OptimalityConditions_Gradient_block}
-\E_{\bm{q}^{\mathcal{B}+}(\hat{\bm{x}},\hat{t})}\big[g_F(\hat{\bm{x}},\omega)\big]\in \mathcal{N}_{\mathcal{X}}(\hat{\bm{x}}).
\end{gather}
\end{subequations}
\end{corollary}

\begin{proof}
The feasibility condition becomes $\sum_{i\in\mathcal{B}} \bar{q}^{\mathcal{B}}_i=\sum_{i\in\mathcal{B}} r_i$ and the stationarity condition follows from first-order optimality for~\eqref{eq:Prox_RegularizedSubproblem_block}.
\end{proof}

\begin{remark}[Size of the block $\mathcal{B}$]
\label{rem:SizeofBlock}
The block size must satisfy $|\mathcal{B}|\ge 2$. Indeed,~\eqref{eq:OptimalityConditions_Feasibility_block} is equivalent to
$\sum_{i\in\mathcal{B}} \bar{q}^{\mathcal{B}}_i=\sum_{i\in\mathcal{B}} r_i$.
If $|\mathcal{B}|=1$, this forces $\bar{q}^{\mathcal{B}}_i=r_i$ for the unique index in the block, hence $\prox_{\gamma\varphi}^{\mathcal{B}}=\operatorname{Id}$ and the update is trivial.
\end{remark}

\section{Convergence analysis}\label{sec:Analysis}

This section establishes convergence guarantees for \Cref{alg:ShortAlg}. 
We begin in \Cref{sub:Complete_algorithm} by presenting the complete
\textsc{EASIeST} scheme, including the outer randomized block Bregman updates
and the inner routines used to solve the regularized subproblems.

The convergence analysis then proceeds in two steps. First, we analyze the
stochastic inner subproblem routine \eqref{eq:tupdate}--\eqref{eq:xupdate} for
a \emph{fixed} dual distribution $\bm q$ and prove sublinear convergence for
general convex objectives in \Cref{th:convergencesubprob}.
Second, we use these inner guarantees to analyze the outer randomized block
Bregman updates in \Cref{alg:ShortAlg}, proving convergence in function values
in \Cref{th:blockproxcovfvalues} and almost sure convergence of the iterates in
\Cref{th:blockproxconvergence}.

\subsection{Complete EASIeST algorithm}
\label{sub:Complete_algorithm}

The derivations in \Crefrange{sub:Bregman_divergences}{sub:OptimalityCond}
lead to the \textsc{EASIeST}, summarized in \Cref{alg:ShortAlg}. At a high level,
\textsc{EASIeST} alternates between:
(i) a \emph{primal} inner stage, which approximately solves a smoothed
CVaR subproblem, and
(ii) a \emph{dual} outer stage, which updates the dual probability vector
$\bm q^k$ by a randomized block Bregman proximal step. We now describe the
two main components of the method.

\paragraph{Step 1. Block generation}
Block-coordinate and randomized coordinate methods have been extensively
studied; see, for example, Chapter~5 of \cite{ryu2022large} and the references
therein. Common block-selection strategies include cyclic rules, randomized
selection from a fixed partition, uniform or nonuniform coordinate sampling
\cite{ShalevShwartzTewari2009,ShalevShwartzTewari2011,Nesterov2012,
RichtarikTakac2014,RichtarikTakac2016,PatrascuNecoara2015}, importance
sampling based on coordinate-wise smoothness constants
\cite{Zhang2004,ZhangXiao2015,LeventhalLewis2010}, and greedy selection rules
\cite{Tseng1990,LuoTseng1992,ChenHeLiZhang2012}.

In our setting, however, a fixed-partition strategy is not appropriate. The block Bregman update preserves the total mass of the selected block
(cf.~\Cref{rem:SizeofBlock}); hence, if the coordinates were partitioned into
fixed blocks, the sequence $\{\bm q^k\}_{k\in\mathbb N}$ would remain confined
to a restricted subset of $\mathcal Q$, which need not contain an optimal dual
solution. We therefore sample a fresh block at every outer iteration.

Specifically, at iteration $k$, we choose a sampling distribution $\bm\pi^k$
over $\{1,\ldots,n\}$ with $\pi_i^k>0$ for all $i$, draw $m_k\ge 2$ indices
without replacement according to $\bm\pi^k$, and denote the resulting block by
$\mathcal B_k$. Since every index has a positive probability of selection, the randomized block updates are not restricted by fixed block-mass constraints and
can explore the full feasible region $\mathcal Q$.

For CVaR optimization, the natural choice is
\[
\bm\pi^k=\bm q^k.
\]
Indeed, the Bregman geometry guarantees
$\bm q^k\in\operatorname{int}\mathcal Q$, so $q_i^k>0$ for all
$i=1,\ldots,n$. Thus, $\bm q^k$ is a valid block-sampling distribution. This
choice is also consistent with the stochastic inner solver, which uses the same
dual probability vector $\bm q^k$ for importance sampling.

\begin{remark}\label{rem:PD_block_sampling}
The choice $\bm\pi^k=\bm q^k$ relies on the fact that the Bregman update keeps
the dual iterates in $\operatorname{int}\mathcal Q$. In contrast, a classical
quadratic proximal update may produce sparse dual iterates with some zero
components. Such iterates are not suitable as block-sampling distributions in
the stochastic setting, since zero-mass scenarios would no longer be sampled.
Although one could introduce a separate distribution $\bm\pi^k\neq\bm q^k$,
this would break the natural link between the outer dual update and the inner
importance-sampling distribution.
\end{remark}

\paragraph{Step 2. Subproblem solution}
At each outer iteration $k$, the block Bregman step requires solving the
regularized subproblem~\eqref{eq:Prox_RegularizedSubproblem_block}. We
distinguish between deterministic and stochastic implementations of
\textsc{EASIeST}. Here, the deterministic setting refers to either the case
where full function values and full (sub)gradients with respect to all scenarios
can be computed, or to the SAA setting, where a fixed finite scenario set is treated as the deterministic objective.

\emph{Deterministic subproblem solution.}
In the deterministic setting, \eqref{eq:Prox_RegularizedSubproblem_block} can be
solved by any appropriate deterministic optimization method. The choice of the
inner solver depends on the structure of $F$, the geometry of $\mathcal X$, the
dimension of the decision variable, and the cost of evaluating $F$ and its
(sub)gradients. For example, if only the convexity of $F(\cdot,\omega)$ is assumed,
one may use a projected subgradient-type method. If $F$ is smooth and the
dimension is moderate, deterministic first-order, quasi-Newton, or second-order methods may be preferable. Thus, the deterministic version of \textsc{EASIeST}
is not tied to a particular inner solver; the convergence analysis only requires
that the regularized subproblems be solved to the prescribed accuracy.

\emph{Stochastic subproblem solution.}
In the stochastic setting, we solve~\eqref{eq:Prox_RegularizedSubproblem_block}
using a projected stochastic subgradient method with $\bm q^k$ as the sampling
distribution. Unlike Step~1, the inner routine uses mini-batch sampling
\emph{with replacement}: scenarios are sampled independently, and repetitions
are allowed. As $\bm q^k$ adapts over the outer iterations, this yields an
adaptive importance-sampling mechanism.

Fix a mini-batch multiset $\mathcal I\subseteq\{1,\ldots,n\}$ with
$|\mathcal I|\ge 1$, where indices in $\mathcal I$ may repeat. For outer
iteration $k$ and inner iteration $j$, we use
\begin{align}
t_k^{j}
&=
\argmin_{t\in\R}\ f^{\mathcal B_k}(\bm x_k^{j},t;\bm q^k),
\label{eq:tupdate}\\
\bm x_k^{j+1}
&=
\operatorname{proj}_{\mathcal X}
\Big[
\bm x_k^{j}
-
\beta_j\,G_{\bm x}(\bm x_k^{j},t_k^{j};\mathcal I)
\Big],
\label{eq:xupdate}
\end{align}
where $\{\beta_j\}_{j\in\mathbb N}$ is a step-size sequence and
$G_{\bm x}(\bm x_k^{j},t_k^{j};\mathcal I)$ is an unbiased mini-batch
stochastic subgradient of $f^{\mathcal B_k}$ with respect to $\bm x$.
Convergence guarantees for~\eqref{eq:tupdate}--\eqref{eq:xupdate} are given
in~\Cref{th:convergencesubprob}.

\subsection{Subproblem convergence for general convex functions}
This subsection analyzes a specific stochastic routine for approximately
solving the regularized subproblems generated by \textsc{EASIeST}. In
particular, we focus on the projected stochastic subgradient scheme
\eqref{eq:tupdate}--\eqref{eq:xupdate} under a fixed dual distribution
$\bm q$. This choice is convenient for the stochastic version of the method,
but it is not intrinsic to the outer block Bregman framework. In deterministic
implementations, or when the subproblem admits additional exploitable structure,
one may replace this inner routine by any solver capable of producing the
required approximate solution.

For a fixed $\bm q \in \mathcal{Q}$, consider the subproblem objective $f^\mathcal{B}(\bm x,t;\bm q)$ and the update rule
\eqref{eq:tupdate}--\eqref{eq:xupdate}, where at iteration $j$ we take
$t^j\in\argmin_{t\in\R} f^\mathcal{B}(\bm x^j,t;\bm q)$ and then perform a projected stochastic (sub)gradient step in $\bm x$.
Throughout this subsection, we view $t^j$ as \emph{exact minimizer} given $\bm x^j$; this allows us to treat the method as projected stochastic subgradient descent on the reduced function
\[
\hat f^\mathcal{B}(\bm x;\bm q)\;:=\;\min_{t\in\R} f^\mathcal{B}(\bm x,t;\bm q),
\qquad\text{so that}\qquad
\hat f^\mathcal{B}(\bm x^j;\bm q)=f^\mathcal{B}(\bm x^j,t^j;\bm q).
\]

\begin{theorem}[Subproblem convergence]\label{th:convergencesubprob}
Let $\{(\bm{x}^j, t^j)\}_{j=0}^{T}$ be generated by \eqref{eq:tupdate}--\eqref{eq:xupdate}.
Assume:
\begin{enumerate}
\item[(i)] (\emph{Bounded second moment}) For all $\bm x\in\mathcal X$ and $t\in\R$,
\[
\mathbb{E}\!\left[\bigl\|G_{\bm{x}}(\bm{x},t;\mathcal{I}) \bigr\|_2^2\right]\;\leq\; B^2.
\]
\item[(ii)] (\emph{Bounded domain}) There exists $R<\infty$ such that $\|\bm x-\bm x^\star\|_2^2\le R^2$ for all $\bm x\in\mathcal X$,
where $\bm x^\star\in\arg\min\limits_{\bm x\in\mathcal X}\hat f^\mathcal{B}(\bm x;\bm q)$.
\end{enumerate}
Define the ergodic averages
$\Bar{\bm{x}}^T := \dfrac{1}{T}\sum\limits_{j=0}^{T-1}\bm{x}^j$ and
$\Bar{t}^T := \dfrac{1}{T}\sum\limits_{j=0}^{T-1}t^j$.
Then
\[
\mathbb{E}\!\left[f^\mathcal{B}(\Bar{\bm{x}}^T,\Bar{t}^T;\bm{q}) - f^\mathcal{B}_*\right]
\;\leq\;
\frac{R^2}{2T\,\beta_{T-1}}
\;+\;
\frac{1}{2T}\sum_{j=0}^{T-1}\beta_j B^2,
\]
where $f^\mathcal{B}_*:=\min\limits_{\bm x\in\mathcal X,\,t\in\R} f^\mathcal{B}(\bm x,t;\bm q)$.
\end{theorem}

\begin{proof}
Let $g^j:=\mathbb{E}\!\left[G_{\bm x}(\bm x^j,t^j;\mathcal I)\mid \bm x^j,t^j\right]$ and
$\eta_j:=G_{\bm x}(\bm x^j,t^j;\mathcal I)-g^j$, so that
$\mathbb{E}[\eta_j\mid \bm x^j,t^j]=0$.
By the nonexpansiveness of the Euclidean projection,
\begin{align*}
\|\bm x^{j+1}-\bm x^\star\|_2^2
&=
\Big\|\proj_{\mathcal X}\!\big(\bm x^j-\beta_j G_{\bm x}(\bm x^j,t^j;\mathcal I)\big)-\bm x^\star\Big\|_2^2\\
&\le
\big\|\bm x^j-\beta_j G_{\bm x}(\bm x^j,t^j;\mathcal I)-\bm x^\star\big\|_2^2\\
&=
\|\bm x^j-\bm x^\star\|_2^2
+\beta_j^2\|G_{\bm x}(\bm x^j,t^j;\mathcal I)\|_2^2
-2\beta_j\langle \bm x^j-\bm x^\star,\, G_{\bm x}(\bm x^j,t^j;\mathcal I)\rangle .
\end{align*}
Taking conditional expectation given $\bm x^j,t^j$ and using $\mathbb{E}[\eta_j\mid \bm x^j,t^j]=0$ gives
\begin{equation*}
\begin{split}
  \mathbb{E}\!\left[\|\bm x^{j+1}-\bm x^\star\|_2^2\mid \bm x^j,t^j\right]
\le
\|\bm x^j-\bm x^\star\|_2^2
&+\beta_j^2\,\mathbb{E}\!\left[\|G_{\bm x}(\bm x^j,t^j;\mathcal I)\|_2^2\mid \bm x^j,t^j\right]\\
&-2\beta_j\langle \bm x^j-\bm x^\star,\, g^j\rangle .  
\end{split}  
\end{equation*}
Since $t^j\in\arg\min_t f^\mathcal{B}(\bm x^j,t;\bm q)$, any expected $\bm x$-(sub)gradient at $(\bm x^j,t^j)$ is a valid subgradient of the reduced function $\hat f^\mathcal{B}(\cdot;\bm q)$ at $\bm x^j$, hence $g^j\in\partial \hat f^\mathcal{B}(\bm x^j;\bm q)$.
By convexity of $\hat f^\mathcal{B}$,
$\langle \bm x^j-\bm x^\star, g^j\rangle\ge \hat f^\mathcal{B}(\bm x^j;\bm q)-\hat f^\mathcal{B}(\bm x^\star;\bm q)= f^\mathcal{B}(\bm x^j,t^j;\bm q)-f^\mathcal{B}_*$.
Using Assumption \emph{(i)} and rearranging yields
\[
\mathbb{E}\!\left[f^\mathcal{B}(\bm x^j,t^j;\bm q)-f^\mathcal{B}_*\right]
\le
\frac{\mathbb{E}\!\left[\|\bm x^j-\bm x^\star\|_2^2\right]-\mathbb{E}\!\left[\|\bm x^{j+1}-\bm x^\star\|_2^2\right]}{2\beta_j}
+\frac{\beta_j}{2}B^2.
\]
Summing over $j=0,\dots,T-1$, using nonincreasing $\beta_j$  to telescope with $\beta_{T-1}$, and then dividing by $T$ gives
\[
\frac{1}{T}\sum_{j=0}^{T-1}\mathbb{E}\!\left[f^\mathcal{B}(\bm x^j,t^j;\bm q)-f^\mathcal{B}_*\right]
\le
\frac{R^2}{2T\,\beta_{T-1}}+\frac{1}{2T}\sum_{j=0}^{T-1}\beta_jB^2.
\]
Finally, Jensen's inequality and convexity of $f^\mathcal{B}(\cdot,\cdot;\bm q)$ imply
$$f^\mathcal{B}(\Bar{\bm x}^T,\Bar t^T;\bm q)\le \frac1T\sum_{j=0}^{T-1} f^\mathcal{B}(\bm x^j,t^j;\bm q),$$ completing the proof.
\end{proof}

\subsection{Convergence of the \textsc{EASIeST}}

We now analyze the outer dual updates of \Cref{alg:ShortAlg}. 
Recall that $\varphi(\bm q)$ denotes the outer objective over $\mathcal Q$, and that $\psi$ is the Fermi--Dirac Legendre function used to define the Bregman divergence $D_\psi$.
We assume the algorithm is initialized at $\bm q^0\in\operatorname{int}\mathcal Q$; by construction of the Bregman proximal step with $\psi$, the iterates remain in $\operatorname{int}\mathcal Q$.

\begin{theorem}[\textsc{EASIeST} convergence in function values]\label{th:blockproxcovfvalues}
Let $\{\bm{q}^k\}_{k\in \mathbb{N}}$ be the sequence of iterates of the \Cref{alg:ShortAlg} with step‐sizes $\{\gamma_k\}_{k\in \mathbb{N}}$ and random blocks $\{\mathcal B_k\}_{k\in \mathbb{N}}$.

If each coordinate $i \in \{1,\ldots,n\}$ is sampled (without replacement) with probability $\pi_i^k > 0$, then the sequence $\{\bm{q}^k\}_{k\in \mathbb{N}}\subset \operatorname{int}\mathcal{Q}$ converges in function values as follows:
  \begin{equation}
    \label{eq:BregmanConvergence_Block}
        \varphi(\bm{q}^k) - \varphi(\bm{q}^\star) \leq \frac{D_\psi(\bm{q}^\star,\bm{q}^0)}{\sum_{j = 0}^{k-1} \gamma_j}
        \,.
    \end{equation}
In particular, if $\sum_{k=0}^\infty \gamma_k=+\infty$, then $\varphi(\bm q^k)\downarrow \varphi(\bm q^\star)$.
\end{theorem}

\begin{proof} 
The proof here follows exactly that of \Cref{thm:BregmanConvergence}, noting that at each iteration $k$ the proximal operator acts only on the coordinate block $\mathcal B_k$. 

\paragraph{Step 1} Since the Bregman divergence $D_\psi$ is additive by definition, denote by $D^{\mathcal B_k}_\psi(\bm{q}^{k+1}, \bm{q}^{k})$ the Bregman divergence between $\bm{q}^k$ and $ \bm{q}^{k+1}$ on the block $\mathcal B_k$. Note that
\[D_\psi(\bm{q}^{k+1}, \bm{q}^{k}) = D^{\mathcal B_k}_\psi(\bm{q}^{k+1}, \bm{q}^{k}) + D^{\mathcal B_k^c}_\psi(\bm{q}^{k+1}, \bm{q}^{k})\, .
\]
Moreover, since $q_i^{k+1} = q_i^k$ for $i \in \mathcal B_k^c$, then $D^{\mathcal B_k^c}_\psi(\bm{q}^k, \bm{q}^{k+1}) = D^{\mathcal B_k^c}_\psi(\bm{q}^k, \bm{q}^{k}) = 0$.
With this, the monotonicity property $\varphi(\bm{q}^{k+1}) \leq \varphi(\bm{q}^{k})$ holds by the same reasoning used in Step 1 of \Cref{app:thm_bregman_convergence}.
\paragraph{Step 2} As in the previous step, observe that the Fermi--Dirac entropy $\psi$ is additive; we denote its value on block $\mathcal B_k$ at $\bm{q}^{k+1}$ by $\psi^{\mathcal B_k}(\bm{q}^{k+1})$. Evidently,
\[\psi(\bm{q}^{k+1}) = \psi^{\mathcal B_k}(\bm{q}^{k+1}) + \psi^{\mathcal B_k^c}(\bm{q}^{k+1}) \, ,
\]
and since $q_i^{k+1} = q_i^k$ for $i \in \mathcal B_k^c$, $\psi^{\mathcal B_k^c}(\bm{q}^{k+1}) = \psi^{\mathcal B_k^c}(\bm{q}^{k})$.

Since each coordinate $i$ of the decision variable $\bm{q}$ has a a strictly positive sampling probability, the sequence of iterates $\{\bm{q}^k\}_{k\in \mathbb{N}}$ is guaranteed to evolve within the $\operatorname{int}\mathcal{Q}$  making $\bm{q}^\star \in \mathcal{Q}$ feasible in the limit.

From here, the proof proceeds as in Step 2 of \Cref{app:thm_bregman_convergence}.
\end{proof}

\begin{theorem}[\textsc{EASIeST} convergence in iterates]\label{th:blockproxconvergence} 
Let $\{\bm{q}^k\}_{k\in \mathbb{N}}\subset \operatorname{int}\mathcal{Q}$ be the sequence of iterates of the \Cref{alg:ShortAlg} with positive step‐sizes $\{\gamma_k \}_{k\in \mathbb{N}}$ such that $\sum_{k=0}^\infty \gamma_k = +\infty$, and random blocks $\{\mathcal{B}_k\}_{k\in \mathbb{N}}$.

Assume that each coordinate $i \in \{1,\ldots,n\}$ is sampled (without replacement) with probability $\pi_i^k > 0$ and each subproblem is solved exactly (cf. Step 2 in \Cref{alg:ShortAlg}). Then
there exists a (random) limit point $\bm{q}^\star\in \argmin\limits_{\bm{q} \in \mathcal{Q}}  \ \varphi(\bm{q})$ such that
\begin{equation}\label{eq;a.s.convergence}
    \|\bm{q}^k - \bm{q}^\star\|_1\;\xrightarrow[k\to\infty]{}\;0
\quad\text{a.s.}
\end{equation}
\end{theorem}

\begin{proof}
We break down the proof into 3 main steps.

\vspace{0.3cm}
\paragraph{(i) Supermartingale structure} First, we show that the stochastic sequence of Bregman divergences $\left\{D_\psi(\bm{q}^\star, \bm{q}^k)\right\}_{k\in \mathbb{N}}$ is quasi-Fej\'er monotone. 
By the three-point identity and the block-optimality condition
\begin{align*}
    D_\psi(\bm{q}^\star,\bm{q}^{k+1}) + D_\psi(\bm{q}^{k+1},\bm{q}^{k})
        &\leq
        D_\psi(\bm{q}^\star,\bm{q}^{k})
        +
        \gamma_k
        \langle
            g_\varphi(\bm{q}^{k+1}),
            \bm{q}^\star - \bm{q}^{k+1}
        \rangle\\
     \text{convexity of } \varphi \quad  & \leq D_\psi(\bm{q}^\star,\bm{q}^{k})
        + \gamma_k (\varphi(\bm{q}^\star) - \varphi(\bm{q}^{k+1}))\\
          \quad &\leq D_\psi(\bm{q}^\star,\bm{q}^{k}) \, .
\end{align*}
Thus, 
\begin{equation}\label{eq:1stepdescent}
    D_\psi(\bm{q}^\star,\bm{q}^{k+1}) \leq D_\psi(\bm{q}^\star,\bm{q}^{k}) - D_\psi(\bm{q}^{k+1},\bm{q}^{k}) 
\end{equation}
Denote $\Delta_k:=D_\psi(\bm{q}^\star,\bm{q}^{k})$, $\delta_k := D_\psi(\bm{q}^{k+1},\bm{q}^{k})$, and $\mathcal{F}_k  := \sigma(\bm{q}^k, \ldots, \bm{q}^0)$ the filtration generated by the sequence if iterates $\{\bm{q}^i\}_{i=0}^k$. Then, taking the conditional expectation of both sides of \eqref{eq:1stepdescent} yields
\begin{align*}
    \mathbb{E}[\Delta_{k+1}|\mathcal{F}_k] &\leq \Delta_k - \mathbb{E}[\delta_k|\mathcal{F}_k]\\
  (\delta_k \geq 0) \quad  &\leq \Delta_k\, . 
\end{align*}
Hence $\Delta_k \geq 0$ is a positive supermartingale and therefore by the Robbins--Siegmund theorem \cite{robbinsSiegmund1985}
\begin{enumerate}
    \item $\Delta_k \;\xrightarrow[k\to\infty]{}\;\Delta_\infty\ \geq 0  \ \text{a.s.}$
    \vspace{0.25cm}
    \item $\sum_{k=0}^\infty \mathbb E[\delta_k\mid\mathcal F_k]<\infty
\quad\Longrightarrow\quad
\delta_k\to 0\ \text{a.s.} \quad\Longrightarrow\quad D_\psi(\bm{q}^{k+1},\bm{q}^k)\;\to\;0
\quad\text{a.s.}$
\end{enumerate}

\vspace{0.3cm}
\paragraph{(ii) Cluster points} Since the iterates $\{\bm{q}^k\}$ lie in the compact set $\mathcal{Q}$, by the Bolzano--Weierstraß theorem there exists a (random) subsequence $\{\bm{q}^{k_j}\}_{j\in \mathbb{N}}$ and a point $\bm{q}^\infty \in \mathcal{Q}$ such that
\[
\bm{q}^{k_j}\;\xrightarrow[j\to\infty]{}\;\bm{q}^\infty
\quad\text{a.s.}
\]
We now show $\bm{q}^\infty \in \argmin\limits_{\bm{q}\in \mathcal{Q}} \varphi(\bm{q})$.

\begin{enumerate}
    \item[(a)] \emph{Convergence in function values.} Since $\sum_{k=0}^\infty \gamma_k = +\infty$, then by \Cref{th:blockproxcovfvalues} $\varphi(\bm{q}^k) \to \varphi(\bm{q}^\star)$ as $k \to \infty$ for any $\bm{q}^\star \in \argmin\limits_{\bm{q}\in \mathcal{Q}} \varphi(\bm{q})$.

    \item[(b)] \emph{The limit.} Since $\varphi(\bm{q}^k) \to \varphi(\bm{q}^\star)$, $\varphi(\bm{q}^{k_j}) \to \varphi(\bm{q}^\star)$ as $j \to \infty$, and so $\varphi(\bm{q}^\infty) \leq \varphi(\bm{q}^\star)$ because $\varphi$ is lower semicontinuous.  
\end{enumerate}
Thus $\bm{q}^\infty \in \argmin\limits_{\bm{q}\in \mathcal{Q}} \varphi(\bm{q})$.
\vspace{0.3cm}
\paragraph{(iii) Bregman--Opial argument and convergence} Now, let us take a  subsequence $\{\bm{q}^{k_j}\}_{j\in \mathbb{N}}$ and a point $\bm{q}^\star \in \mathcal{Q}$ such that
\[
\bm{q}^{k_j}\;\xrightarrow[j\to\infty]{}\;\bm{q}^\star
\quad\text{a.s.}
\]
From part \textit{(ii)}, we have that $\bm{q}^\star \in \argmin\limits_{\bm{q}\in \mathcal{Q}} \varphi(\bm{q})$. Then, since $\mathcal{Q}$ is a polytope, \cite[Theorem 1]{PAUWELS2024107183} implies that 
\[D_\psi(\bm{q}^\star,\bm{q}^{k_j}) \;\xrightarrow[j\to\infty]{}\; D_\psi(\bm{q}^\star,\bm{q}^\star) = 0 \text{ a.s.}
\]
Therefore, since $\Delta_k \to \Delta_\infty$ by \textit{(i)} and $D_\psi(\bm{q}^\star,\bm{q}^{k_j}) \to 0$, we conclude that 
\[D_\psi(\bm{q}^\star,\bm{q}^{k}) \to 0 \text{ a.s. as } k \to \infty\, .
\] 
Finally, invoking \Cref{lem:gen-pinsker} yields 
\[
 \|\bm{q}^k - \bm{q}^\star\|_1\;\xrightarrow[k\to\infty]{}\;0
\quad\text{a.s.}
\]
which concludes the proof.
\end{proof}

\section{Practical implementation}\label{sec:PracticalImplementation}
This section discusses the practical implementation of~\Cref{alg:ShortAlg}. We consider two settings: \emph{(i) deterministic}, where CVaR, smoothed CVaR, and their respective (sub)gra\-dients are computed either exactly or via sample average approximation (SAA); and \emph{(ii) stochastic}, where only stochastic (e.g., mini-batch) estimates of the relevant quantities are available. Our focus here is on the practical implementation of the proposed method. For completeness, the deterministic and stochastic baseline algorithms used for performance comparisons in~\Cref{sec:NumericalExperiments}, together with the corresponding CVaR subgradient constructions, are deferred to~\Cref{app:baselines}.

\subsection{Deterministic setting}\label{sub:DeterministicSetting}
We first describe the deterministic implementation of \textsc{EASIeST}. For numerical comparisons, we also use a deterministic baseline subgradient method with adaptive stepsize control based on~\cite{uryasev1983adaptive}; its pseudocode and the explicit deterministic CVaR subgradient formula are given in~\Cref{app:det_baseline}.

\paragraph{Inexact Bregman proximal point method}
Before presenting \Cref{alg:inexact-easi-det}, we discuss a key practical aspect: inexactness.
Recall $\varphi(\bm{q})$ defined in \eqref{eq:phi_for_sub}. 
The (exact) Bregman proximal point method computes $\mathbold{q}^{k+1}$ by solving
\[
\mathbf{0} \in \partial \varphi(\mathbold{q}^{k+1}) + \frac{1}{\gamma_k}\big(\nabla\psi(\mathbold{q}^{k+1})-\nabla\psi(\mathbold{q}^{k})\big),
\]
equivalently,
\[
\mathbold{q}^{k+1}\in\argmin_{\bm{q}}\Big\{\varphi(\bm{q})+\frac{1}{\gamma_k}D_\psi(\bm{q},\mathbold{q}^k)\Big\} \, ,
\]
which in our case leads to the subproblem \eqref{eq:Prox_RegularizedSubproblem} (or \eqref{eq:Prox_RegularizedSubproblem_block}). 
In practice, this subproblem is rarely solved exactly, and one therefore replaces the exact optimality condition by an \emph{inexact} one. Two widely used criteria are due to Eckstein \cite{eckstein1998approximate} and Solodov--Svaiter \cite{solodov2000inexact}.

\paragraph{Eckstein-type inexactness}
Eckstein models the inexactness as an additive error in the Bregman optimality relation:
\[
\mathbold{e}^k + \nabla\psi(\mathbold{q}^k)
= \nabla\psi(\mathbold{q}^{k+1}) + \gamma_k \mathbold{g}^{k+1},
\qquad \mathbold{g}^{k+1}\in\partial\varphi(\mathbold{q}^{k+1}),
\]
together with summability conditions on the error sequence, e.g.,
$\sum_{k=1}^\infty \|\mathbold{e}^k\|_2<\infty$ and
$\sum_{k=1}^\infty \langle \mathbold{e}^k,\mathbold{q}^k\rangle<\infty$.
Intuitively, $\mathbold{e}^k$ quantifies the residual in the dual (mirror) variable
$\nabla\psi(\mathbold{q})$: the method behaves like the exact Bregman proximal point method up to a perturbation whose total accumulated effect is finite, which is sufficient for convergence of the outer iterates under standard assumptions.

\paragraph{Solodov--Svaiter relative error criterion}
Solodov and Svaiter propose a constructive \emph{relative} termination rule based on the Bregman geometry.
Given a candidate subgradient $\mathbold{g}^{k+1}\in\partial\varphi(\mathbold{q}^{k+1})$, define the
\emph{exact} Bregman-prox point associated with $(\mathbold{q}^k,\mathbold{g}^{k+1})$ as
\[
\mathbold{r}
:=\nabla\psi^{-1}\Big(\nabla\psi(\mathbold{q}^k)-\gamma_k \mathbold{g}^{k+1}\Big).
\]
Then the inexactness requirement is
\[
D_\psi(\mathbold{q}^{k+1},\mathbold{r}) \;\le\; \rho^2\,D_\psi(\mathbold{q}^{k+1},\mathbold{q}^k),
\qquad \rho\in(0,1).
\]
This condition enforces that $\mathbold{q}^{k+1}$ is significantly closer (in Bregman distance) to the
ideal mirror step $\mathbold{r}$ than to the previous iterate $\mathbold{q}^k$, i.e., the inner solve makes
\emph{relative progress} measured in the same divergence that defines the proximal regularization.
As a result, one obtains a globally convergent inexact Bregman proximal point method without requiring absolute summability of residuals or preconstructed error sequences, making the criterion appealing for stopping rules in iterative inner solvers.

\paragraph{A computable Solodov--Svaiter-type criterion}
We adopt a Solodov--Svaiter-type \emph{relative} rule, but the original criterion is not directly implementable
in our setting. Indeed, $\varphi$ is defined through an inner minimization (cf.~\eqref{eq:phi_for_sub}),
\[
\varphi(\bm q)=-\min_{\bm x\in\cX} \E_{\bm q}[F(\bm x,\omega)].
\]
Consequently, by Danskin's theorem, see e.g.,  \cite{danskin1967, bertsekas1999nonlinear}, evaluating (or selecting) $\mathbold{g}^{k+1}\in\partial\varphi(\mathbold{q}^{k+1})$
requires access to a minimizer of the auxiliary problem
\[
\min_{\bm x\in\cX}\ \E_{\bm q}[F(\bm x,\omega)],
\]
which would either force an (almost) exact inner solve or introduce an additional auxiliary optimization that we aim
to avoid.

To quantify the effect of reusing an inner solution obtained for a nearby probability vector, define\footnote{assuming $F(\cdot,\omega)$ is differentiable for all $\omega \in \Omega$.}
\[
f_{\bm q}(\bm x)\coloneqq \E_{\bm q}\!\big[F(\bm x,\omega)\big],
\qquad
\nabla f_{\bm q}(\bm x)=\E_{\bm q}\!\big[\nabla_{\bm x}F(\bm x,\omega)\big],
\]
and, for any $\beta>0$, the projected gradient mapping
\[
\mathcal{G}_\beta(\bm x;\bm q)
\coloneqq
\frac{1}{\beta}\Big(\bm x-\operatorname{proj}_{\cX}\big(\bm x-\beta\,\nabla f_{\bm q}(\bm x)\big)\Big),
\]
where a small $\|\mathcal{G}_\beta(\bm x;\bm q)\|_2$ certifies near-stationarity of $\bm x$ for
$\min_{x\in\cX} f_{\bm q}(x)$.
The following result quantifies how $\|\mathcal{G}_\beta(\bm x;\bm q)\|_2$ changes when the probability vector $\bm q$
is perturbed.

\begin{proposition}\label{prop:PG_stability_change_measure}
Let $F:\cX\times\Omega\to\mathbb{R}$ be differentiable in $\bm x$ for all $\omega\in\Omega$.
Assume that the scenario-gradient is essentially bounded uniformly over $\cX$, i.e., there exists
$G_{\cX}<\infty$ such that
\begin{equation}\label{eq:esssup_grad_bound}
\operatorname*{ess\,sup}_{\omega\in\Omega}\ \|\nabla_{\bm x}F(\bm x,\omega)\|_2 \le G_{\cX},
\qquad \forall\,\bm x\in\cX.
\end{equation}
Then, for any $\bm q,\tilde{\bm q}\in\mathcal{Q}$, any $\beta>0$, and any $\bm x\in\cX$,
\begin{equation}\label{eq:PG_stability}
\big\|\mathcal{G}_\beta(\bm x;\bm q)\big\|_2
\le
\big\|\mathcal{G}_\beta(\bm x;\tilde{\bm q})\big\|_2
+
G_{\cX}\sqrt{2D_\psi(\bm q,\tilde{\bm q})}.
\end{equation}
\end{proposition}

\begin{proof}
By the triangle inequality,
\[
\|\mathcal{G}_\beta(\bm x;\bm q)\|_2
\le
\|\mathcal{G}_\beta(\bm x;\tilde{\bm q})\|_2
+
\|\mathcal{G}_\beta(\bm x;\bm q)-\mathcal{G}_\beta(\bm x;\tilde{\bm q})\|_2.
\]
Using nonexpansiveness of the Euclidean projection $\proj_{\cX}$,
\begin{align*}
\|\mathcal{G}_\beta(\bm x;\bm q)-\mathcal{G}_\beta(\bm x;\tilde{\bm q})\|_2
&=
\frac{1}{\beta}\Big\|
\proj_{\cX}\big(\bm x-\beta\nabla f_{\bm q}(\bm x)\big)
-
\proj_{\cX}\big(\bm x-\beta\nabla f_{\tilde{\bm q}}(\bm x)\big)
\Big\|_2 \\
&\le
\frac{1}{\beta}\Big\|
\big(\bm x-\beta\nabla f_{\bm q}(\bm x)\big)
-
\big(\bm x-\beta\nabla f_{\tilde{\bm q}}(\bm x)\big)
\Big\|_2 \\
&=
\|\nabla f_{\bm q}(\bm x)-\nabla f_{\tilde{\bm q}}(\bm x)\|_2.
\end{align*}
In the finite-scenario setting, $\nabla f_{\bm q}(\bm x)=\sum_{i=1}^n q_i\nabla F_i(\bm x)$, hence
\[
\nabla f_{\bm q}(\bm x)-\nabla f_{\tilde{\bm q}}(\bm x)
=
\sum_{i=1}^n (q_i-\tilde q_i)\nabla F_i(\bm x).
\]
Therefore, using $\|\nabla F_i(\bm x)\|_2\le \max_{j}\|\nabla F_j(\bm x)\|_2 \le G_{\cX}$,
\[
\|\nabla f_{\bm q}(\bm x)-\nabla f_{\tilde{\bm q}}(\bm x)\|_2
\le
\sum_{i=1}^n |q_i-\tilde q_i|\,\|\nabla F_i(\bm x)\|_2
\le
G_{\cX}\|\bm q-\tilde{\bm q}\|_1.
\]
Finally, by \Cref{lem:gen-pinsker}, $\|\bm q-\tilde{\bm q}\|_1 \le \sqrt{2D_\psi(\bm q,\tilde{\bm q})}$, which yields
\[
\|\mathcal{G}_\beta(\bm x;\bm q)\|_2
\le
\|\mathcal{G}_\beta(\bm x;\tilde{\bm q})\|_2
+
G_{\cX}\sqrt{2D_\psi(\bm q,\tilde{\bm q})}.
\]
\end{proof}
\begin{remark}\label{rem:PG_stability_l2}
Under the assumptions of \Cref{prop:PG_stability_change_measure}, one also has
\[
\|\mathcal G_\beta(\bm x;\bm q)\|_2
\le
\|\mathcal G_\beta(\bm x;\tilde{\bm q})\|_2
+
G_{\cX}\sqrt{n}\,\|\bm q-\tilde{\bm q}\|_2,
\]
for any \(\bm q,\tilde{\bm q}\in\mathcal Q\), any \(\beta>0\), and any \(\bm x\in\cX\). This follows from the estimate
\[
\|\nabla f_{\bm q}(\bm x)-\nabla f_{\tilde{\bm q}}(\bm x)\|_2
\le
G_{\cX}\|\bm q-\tilde{\bm q}\|_1
\le
G_{\cX}\sqrt{n}\,\|\bm q-\tilde{\bm q}\|_2.
\]
Such a bound is important for controlling inexactness in the classical proximal point setup \cite{rockafellar1976monotone, rockafellar1976augmented}, which leads to quadratic epi-regularization of the CVaR \cite{kouri2022primal}. Note that the above bound depends explicitly on \(n\). Consequently, the resulting estimate deteriorates with the dimension, which makes the bound in \Cref{prop:PG_stability_change_measure} based on the Bregman divergence more attractive.
\end{remark}

At outer iteration $k$, applying the original Solodov--Svaiter (SS) test at an approximate proximal point $\bm q$
would require a subgradient $\bm g\in\partial\varphi(\bm q)$ and thus, by \eqref{eq:phi_for_sub}, solving the auxiliary
problem $\min_{\bm x\in\cX}\E_{\bm q}[F(\bm x,\omega)]$. To avoid this extra solve, we generate two successive candidate
pairs $(\bm q_1,\bm x_1)$ and $(\bm q_2,\bm x_2)$ from the same inner routine, where $(\bm q_2,\bm x_2)$ is a refinement
of $(\bm q_1,\bm x_1)$, and we accept $(\bm q_1,\bm x_2)$ only when two \emph{simultaneous} conditions hold.

\smallskip
\noindent\textit{(i) Modified SS condition.}
We enforce the computable SS surrogate
\[
D_\psi(\bm q_1,\bm q_2)\ \le\ \rho_k^2\,D_\psi(\bm q_1,\bm q^k),
\qquad \rho_k\in(0,1),\ \rho_k\downarrow 0,
\]
which ensures that the refinement $\bm q_2$ is relatively close to $\bm q_1$ compared to the outer step size.

\smallskip
\noindent\textit{(ii) Gradient (stationarity) condition.}
We also require that the refined inner point $\bm x_2$ satisfies the projected-gradient bound
\[
\big\|\mathcal{G}_\beta(\bm x_2;\bm q_2)\big\|_2
\leq G_{\cX} \rho_k\sqrt{2D_\psi(\bm q_1,\bm q^k)}.
\]
By Proposition~\ref{prop:PG_stability_change_measure}, this implies
\[
\big\|\mathcal{G}_\beta(\bm x_2;\bm q_1)\big\|_2
\ \le\
\big\|\mathcal{G}_\beta(\bm x_2;\bm q_2)\big\|_2
+G_{\cX}\sqrt{2D_\psi(\bm q_1,\bm q_2)}
\ \le\
2G_{\cX}\rho_k\sqrt{2\,D_\psi(\bm q_1,\bm q^k)}.
\]
Hence, if $\rho_k\to 0$, the projected-gradient residual associated with the auxiliary problem
$\min_{x\in\cX}\E_{\bm q_1}[F(x,\omega)]$ is driven to zero \emph{asymptotically}, recovering the auxiliary optimality
conditions without explicitly solving that problem.

\begin{algorithm2e}[h!]
\DontPrintSemicolon
\caption{\label{alg:inexact-easi-det}\textsc{EASIeST} for CVaR Optimization (``deterministic'')}
\SetKwInOut{Input}{Input}
\SetKwInOut{Params}{Parameters}
\SetKwInOut{Output}{Output}
\SetKwInOut{Define}{Define}

\Input{\( \bm{x}^{0} \in \cX\) and \( \bm{s}^{0}\in\mathbb{R}^n\), $\bm q^{0} = \bm \sigma(\bm s^{0})$.}
\Params{\( \gamma_0,\beta>0, \ c_\gamma>1\), \( \varepsilon_g, \varepsilon_{TV}, \varepsilon_q >0\), \( \{\rho_k\}\in(0,1)\), $|\mathcal{B}|>1$, \(K,J\in\mathbb{N}\).}
\Define{
Sigmoid map: \( \sigma_i(s)\coloneqq \dfrac{p_i}{1-\alpha}\,\dfrac{e^{s}}{1+e^{s}} \);
vector form \( \bm{\sigma}(\bm{s}) \).\newline
}
\Output{\( \bm{x}^\star \in \cX\), \( \bm{q}^\star=\bm{\sigma}(\bm{s}^\star) \in \mathcal{Q}\), and \( \CVaR_\alpha(F(\bm{x}^\star,\omega))=\mathbb{E}_{\bm{q}^\star}[F(\bm{x}^\star,\omega)]\).}

\BlankLine
\For(\tcp*[f]{Outer loop}){$k=0,1,2,\ldots,K$}{
\BlankLine
\emph{Step 0.} Initialize $\bm x^{k,0} = \bm x^k,  \ \bm s^{k,0} = \bm s^k, \  \bm q^{k,0} = \bm q^k;$

\BlankLine
\emph{Step 1.}  Generate $\mathcal{B}_k = \operatorname{BlockSampler}(\bm q^k, |\mathcal{B}|)$ and set  $\delta_k = \sum_{i \in \mathcal{B}_k} q^k_i;$

\BlankLine
\For(\tcp*[f]{Inner loop}){$j=0,1,2,\ldots,J$}{
\BlankLine
\emph{Step 2.} Set $(\bm s^{k,j+1},\bm q^{k,j+1})=\operatorname{ProxUpdate}(\bm x^{k,j},\bm s^k,\bm q^k,\mathcal B_k,\gamma_k,\delta_k)$;\;

\BlankLine
\emph{Step 3.} Set $\mathcal{I}_j = \{i:q^{k,j+1}_i\geq\varepsilon_q\}$ and evaluate $\{\nabla F_i(\bm x^{k,j})\}_{i\in \mathcal{I}_j};$\;

\BlankLine
\emph{Step 4.} Compute  \( \bm{g}^{k,j} = \sum\limits_{i \in \mathcal{I}_j} q^{k,j+1}_i \,\nabla F_i(\bm{x}^{k,j});\)\;

\BlankLine
\emph{Step 5.} Set $G_\cX^j = \max\limits_{i \in \mathcal{I}_j} \|\nabla F_i(\bm{x}^{k,j})\|_2$ and $D_j = \rho_k^2D_\psi( \bm{q}^{k,j}, \bm{q}^{k});$

\BlankLine
\If(){$ \|\mathcal{G}_\beta(\bm x^{k,j}; \bm q^{k,j+1})\|_2\le G^j_\cX\sqrt{2D_j} \ \ \textbf{and}\ \ D_\psi\!\left(\bm{q}^{k,j},\bm{q}^{k,j+1}\right) \le D_j$}{
\textbf{break}
}

\BlankLine
\emph{Step 6.} Set $\bm{x}^{k,j+1}  =  \operatorname{Update}(\bm{x}^{k,j}, \bm{g}^{k,j},\textrm{solver parameters});$ \;
}

\BlankLine
\emph{Step 7.} If the inner loop terminates at \emph{Step 5} on index \(j\), set \( \bm{s}^{k+1}  =  \bm{s}^{k,j+1}\), \( \bm{q}^{k+1} = \bm{q}^{k,j+1}\), and \( \bm{x}^{k+1} = \bm{x}^{k, j}\); otherwise, set \( \bm{s}^{k+1}  =  \bm{s}^{k,j+1}\), \( \bm{q}^{k+1} = \bm{q}^{k,j+1}\), and \( \bm{x}^{k+1} = \bm{x}^{k, j+1}\);

\BlankLine
\emph{Step 8.} Update $\gamma_{k+1} = c_{\gamma}\gamma_k;$\;

\BlankLine
\If(){$ \|\mathcal{G}_\beta(\bm x^{k+1}; \bm q^{k+1})\|_2 \le \varepsilon_g \ \ \textbf{and}\ \ \frac{1}{2}\|\bm q^{k+1} - \bm q^{k}\|_1 \le \varepsilon_{TV}$}{
\textbf{break}
}
}
\Return \( \bm{x}^\star \gets \bm{x}^{k+1}, \ \bm{q}^\star \gets \bm{q}^{k+1}\), and \( \CVaR_\alpha(F(\bm{x}^\star,\omega))=\mathbb{E}_{\bm{q}^\star}[F(\bm{x}^\star,\omega)]\).
\end{algorithm2e}

\Cref{alg:inexact-easi-det} is the practical version of \Cref{alg:ShortAlg} in the deterministic setting.
Each outer iteration $k$ samples a block $\mathcal{B}_k$ of scenarios according to the current
dual distribution $\bm q^k$ and updates only the corresponding \emph{logits}
$s^k_i:= \nabla \psi (q^k_i)$ for $i \in \mathcal{B}_k$, while keeping $q^k_i$ fixed for $i\notin\mathcal{B}_k$ and preserving the block mass
$\delta_k=\sum_{i\in\mathcal{B}_k}q_i^k$ via a scalar shift (cf. \Cref{alg:block_prox_update}).
The primal variable is then updated by a \emph{user-specified} routine $\bm x^{k,j+1}=\operatorname{Update}(\bm x^{k,j},\bm g^{k,j},\ldots)$ (e.g., projected gradient, quasi-Newton, or any convex solver), so the algorithm can leverage problem structure and available inner solvers.
Inexactness is controlled by two \emph{simultaneous} stopping tests: a computable Solodov--Svaiter-type condition
based on two consecutive dual iterates w.r.t. the Bregman distance and a projected-gradient stationarity test in $\bm x$.

\begin{remark}\label{rem:prox_stationarity}
The definition of $\mathcal{G}_\beta(\bm x;\bm q)$ above assumes that $F(\cdot,\omega)$ is differentiable, so that
$\nabla f_{\bm q}(\bm x)=\E_{\bm q}[\nabla_{\bm x}F(\bm x,\omega)]$ is well-defined. If $F(\cdot,\omega)$ is
convex but possibly \emph{nondifferentiable}, we replace the projected-gradient mapping by a proximal-gradient
(resolvent) stationarity measure.

Fix $\beta>0$ and select any measurable subgradient field
$\bm g_F(\bm x,\omega)\in \partial_{\bm x}F(\bm x,\omega)$. Define
\[
\partial f_{\bm q}(\bm x)=\E_{\bm q}\!\big[\partial_{\bm x}F(\bm x,\omega)\big]
\quad\text{and pick}\quad
\bm v_{\bm q}(\bm x)\in \partial f_{\bm q}(\bm x)
\ \ \text{(e.g., } \bm v_{\bm q}(\bm x)=\E_{\bm q}[\bm g_F(\bm x,\omega)]\text{)}.
\]
We then define the \emph{proximal stationarity mapping}
\[
\mathcal{P}_\beta(\bm x;\bm q)
\coloneqq
\frac{1}{\beta}\Big(\bm x-\operatorname{prox}_{\beta\,\delta_{\cX}}\big(\bm x-\beta\,\bm v_{\bm q}(\bm x)\big)\Big)
=
\frac{1}{\beta}\Big(\bm x-\operatorname{proj}_{\cX}\big(\bm x-\beta\,\bm v_{\bm q}(\bm x)\big)\Big),
\]
where $\delta_{\cX}$ is the indicator of $\cX$ and $\operatorname{prox}_{\beta\,\delta_{\cX}}=\operatorname{proj}_{\cX}$.
In particular, $\mathcal{P}_\beta(\bm x;\bm q)=\bm 0$ if and only if
\[
\bm 0 \in \partial f_{\bm q}(\bm x)+\mathcal{N}_{\cX}(\bm x),
\]
i.e., $\bm x$ is a (first-order) stationary point of $\min_{\bm x\in\cX} f_{\bm q}(\bm x)$ in the standard
convex-analytic sense. When $f_{\bm q}$ is differentiable, choosing $\bm v_{\bm q}(\bm x)=\nabla f_{\bm q}(\bm x)$
recovers the projected-gradient mapping $\mathcal{G}_\beta$.
\end{remark}

\subsection{Stochastic setting}
We next consider the stochastic setting, where only mini-batch estimates of the relevant quantities are available. For numerical comparisons, we also use a stochastic adaptive-subgradient baseline based on~\cite{uryasev1983adaptive}; its pseudocode and the corresponding mini-batch CVaR subgradient construction are deferred to~\Cref{app:stoch_baseline}.
\begin{algorithm2e}[h!]
\DontPrintSemicolon
\caption{\label{alg:inexact-easi-stoch}\textsc{EASIeST} for CVaR Optimization (``stochastic'')}
\SetKwInOut{Input}{Input}
\SetKwInOut{Params}{Parameters}
\SetKwInOut{Output}{Output}
\SetKwInOut{Define}{Define}

\Input{\( \bm{x}^{0} \in \cX\) and \( \bm{s}^{0}\in\mathbb{R}^n\), $\bm q^{0} = \bm \sigma(\bm s^{0})$.}
\Params{\( \gamma_0,\beta>0, \ c_\gamma>1\), \( \varepsilon_g, \varepsilon_{TV}, \varepsilon_q >0\), \( \{\rho_k\}\in(0,1)\), $|\mathcal{B}|>1$, \(K,J, m\in\mathbb{N}\).}
\Define{
Sigmoid map: \( \sigma_i(s)\coloneqq \dfrac{p_i}{1-\alpha}\,\dfrac{e^{s}}{1+e^{s}} \);
vector form \( \bm{\sigma}(\bm{s}) \).\newline
}
\Output{\( \bm{x}^\star \in \cX\), \( \bm{q}^\star=\bm{\sigma}(\bm{s}^\star) \in \mathcal{Q}\), and \( \CVaR_\alpha(F(\bm{x}^\star,\omega))=\mathbb{E}_{\bm{q}^\star}[F(\bm{x}^\star,\omega)]\).}

\BlankLine
\For(\tcp*[f]{Outer loop}){$k=0,1,2,\ldots,K$}{
\BlankLine
\emph{Step 0.} Initialize $\bm x^{k,0} = \bm x^k,  \ \bm s^{k,0} = \bm s^k, \  \bm q^{k,0} = \bm q^k;$

\BlankLine
\emph{Step 1.}  Generate $\mathcal{B}_k = \operatorname{BlockSampler}(\bm q^k, |\mathcal{B}|)$ and set $\delta_k = \sum_{i \in \mathcal{B}_k} q^k_i;$

\BlankLine
\For(\tcp*[f]{Inner loop}){$j=0,1,2,\ldots,J$}{
\BlankLine
\emph{Step 2.} Sample indices $\xi_1,\ldots,\xi_m \stackrel{\mathrm{iid}}{\sim} \bm q^k$ and construct index sets $\Xi_j = \{\xi_1,\ldots,\xi_m\}, \ \Xi_j^{in} = \{i \in \Xi_j: i \in \mathcal{B}_k \cap \Xi_j\}, \ \Xi_j^{out} = \{i \in \Xi_j: i \notin \mathcal{B}_k \cap \Xi_j\}$ and $w_i = \frac{\# \text{ of times index } i \in \Xi_j \text{ repeats}}{|\Xi_j|};$\;

\BlankLine
\emph{Step 3.} Set $(\bm s^{k,j+1},\bm q^{k,j+1})=\operatorname{ProxUpdate}(\bm x^{k,j},\bm s^{k},\bm q^k,\mathcal B_k,\gamma_k,\delta_k)$;\;

\BlankLine
\emph{Step 4.} Set $\mathcal{I}_j = \{i \in \Xi_j^{in}:q^{k,j+1}_i\geq\varepsilon_q\}$ and evaluate $\{\nabla F_i(\bm x^{k,j})\}_{i\in \mathcal{I}_j \cup \Xi_j^{out}};$\;

\BlankLine
\emph{Step 5.} Set  $\bm{g}^{k,j} = \sum\limits_{i \in \Xi_j^{out}} w_i \,\nabla F_i(\bm{x}^{k,j}) + \sum\limits_{i \in \mathcal{I}_j} w_i \frac{q_i^{k,j+1}}{q_i^{k}}\,\nabla F_i(\bm{x}^{k,j});$\;

\BlankLine
\emph{Step 6.} If $|\mathcal{I}_j| \neq 0$ define $\mathcal{S} = \mathcal{I}_j$ else $\mathcal{S} = \Xi_j^{out}$. Set $G_\cX^j = \max\limits_{i \in \mathcal{S}} \|\nabla F_i(\bm{x}^{k,j})\|_2$ and $D_j = \rho_k^2D_\psi( \bm{q}^{k,j}, \bm{q}^{k});$

\BlankLine
\If(){$ \|\hat{\mathcal{G}}_\beta^{(m)}(\bm x^{k,j}; \bm q^{k,j+1}\mid \bm q^k)\|_2\le G^j_\cX\sqrt{2D_j} \ \ \textbf{and}\ \ D_\psi\!\left(\bm{q}^{k,j},\bm{q}^{k,j+1}\right) \le D_j$}{
\textbf{break}
}

\BlankLine
\emph{Step 7.} Set $\bm{x}^{k,j+1}  =  \operatorname{Update}(\bm{x}^{k,j}, \bm{g}^{k,j},\textrm{solver parameters});$ \;
}

\BlankLine
\emph{Step 8.}   If the inner loop terminates at \emph{Step 6} on index \(j\), set \( \bm{s}^{k+1}  =  \bm{s}^{k,j+1}\), \( \bm{q}^{k+1} = \bm{q}^{k,j+1}\), and \( \bm{x}^{k+1} = \bm{x}^{k, j}\); otherwise, set \( \bm{s}^{k+1}  =  \bm{s}^{k,j+1}\), \( \bm{q}^{k+1} = \bm{q}^{k,j+1}\), and \( \bm{x}^{k+1} = \bm{x}^{k, j+1}\);

\BlankLine
\emph{Step 9.} Update $\gamma_{k+1} = c_{\gamma}\gamma_k;$\;

\BlankLine
\If(){$ \|\hat{\mathcal{G}}^{(m)}_\beta(\bm x^{k+1}; \bm q^{k+1}\mid\bm q^k)\|_2 \le \varepsilon_g \ \ \textbf{and}\ \ \frac{1}{2}\|\bm q^{k+1} - \bm q^{k}\|_1 \le \varepsilon_{TV}$}{
\textbf{break}
}
}
\Return \( \bm{x}^\star \gets \bm{x}^{k+1}, \ \bm{q}^\star \gets \bm{q}^{k+1}\), and \( \CVaR_\alpha(F(\bm{x}^\star,\omega))=\mathbb{E}_{\bm{q}^\star}[F(\bm{x}^\star,\omega)]\).
\end{algorithm2e}

To define the stochastic stationarity tests used in \Cref{alg:inexact-easi-stoch}, let $\xi_1,\ldots,\xi_m \stackrel{\mathrm{iid}}{\sim} \bm q$ and define the mini-batch estimator
\[
\widehat{\nabla} f_{\bm q}^{(m)}(\bm x)
\;\coloneqq\;
\frac{1}{m}\sum_{r=1}^m \nabla F_{\xi_r}(\bm x),
\qquad\text{so that}\qquad
\mathbb E\!\left[\widehat{\nabla} f_{\bm q}^{(m)}(\bm x)\right]=\nabla f_{\bm q}(\bm x).
\]
We then define the mini-batch stochastic reduced (projected-gradient) mapping as
\[
\widehat{\mathcal G}_{\beta}^{(m)}(\bm x;\bm q)
\;\coloneqq\;
\frac{1}{\beta}\Big(\bm x-\proj_{\cX}\big(\bm x-\beta\,\widehat{\nabla} f_{\bm q}^{(m)}(\bm x)\big)\Big).
\]
If the batch is drawn from a reference distribution $\bm r\in\mathcal Q$ (e.g., fixed throughout an inner loop), an unbiased estimator of $\nabla f_{\bm q}(\bm x)$ is
\[
\widehat{\nabla} f_{\bm q}^{(m)}(\bm x\mid \bm r)
\;\coloneqq\;
\frac{1}{m}\sum_{r=1}^m \frac{q_{\xi_r}}{r_{\xi_r}}\,\nabla F_{\xi_r}(\bm x),
\]
and the corresponding certificate is
\[
\widehat{\mathcal G}_{\beta}^{(m)}(\bm x;\bm q\mid \bm r)
\;\coloneqq\;
\frac{1}{\beta}\Big(\bm x-\proj_{\cX}\big(\bm x-\beta\,\widehat{\nabla} f_{\bm q}^{(m)}(\bm x\mid \bm r)\big)\Big).
\]

\Cref{alg:inexact-easi-stoch} is the practical version of \Cref{alg:ShortAlg} in the stochastic setting. It mirrors the deterministic scheme, but replaces exact gradient information by mini-batch estimators and uses an importance-weighted stationarity certificate relative to the reference distribution $\bm q^k$. As in the deterministic case, the method combines a computable Solodov--Svaiter-type stopping condition in the dual variable with a stationarity test for the primal update.

\section{Numerical experiments}\label{sec:NumericalExperiments}
This section reports numerical experiments evaluating \Cref{alg:inexact-easi-det} and \Cref{alg:inexact-easi-stoch}. We benchmark their performance against the corresponding deterministic and
stochastic subgradient baselines with adaptive step sizes based on
\cite{uryasev1983adaptive}; see \Cref{alg:BaselineDet,alg:BaselineStoch} in
\Cref{app:baselines}. In the deterministic setting, we additionally compare against Shor’s $r$-algorithm \cite{Shor1998Nondifferentiable} and with the \textsc{EASIeST} variant corresponding to the
classical proximal point method proposed in \cite{kouri2022primal}, obtained
by replacing the Bregman divergence term $D_\psi(\bm q,\bm r)$ with the
quadratic term $\frac{1}{2}\|\bm q-\bm r\|_2^2$. Because the underlying optimization problem is nonsmooth, we measure computational effort by \emph{first-order oracle complexity}, i.e., the total number of oracle calls returning a function value and a (sub)gradient. Additionally, we report the suboptimality gap
\[
f(\bm x^k)-f^\star,
\]
where \(f^\star\) denotes the reference optimal value, versus the cumulative
number of per-scenario function evaluations and gradient evaluations separately.
This distinction is important for certain classes of problems, such as
PDE-constrained optimization.
\subsection{Support Vector Classification}\label{subsec:SVC}
We consider binary classification with training samples\\ \(\{(y_i, \bm z_i)\}_{i=1}^n\), where
$\bm z_i \in\mathbb{R}^d$ are feature vectors and $y_i\in\{-1,+1\}$ are class labels.
Let the \emph{extended} feature vector be $\bar{\bm z}_i\coloneqq (1, \bm z_i^\top)^\top\in\mathbb{R}^{d+1}$, then define the decision vector
\[
\bm x \coloneqq (x_0,\ldots, x_d)^\top\in\mathbb{R}^{d+1}.
\]
Our goal is to learn a linear decision rule parameterized by $\bm x\in\mathbb{R}^{d+1}$,
\[
a(\bm x,\bar{\bm z}) \coloneqq \operatorname{sign}(\bm x^\top \bar{\bm z}),
\]
which assigns to each observation $i$ with extended feature vector $\bar{\bm z}_i\in\mathbb{R}^{d+1}$ a predicted class
label $a(\bm x,\bar{\bm z}_i)\in\{-1,+1\}$.
For each sample $i$, we define the signed margin loss
\[
F_i(\bm x)\coloneqq -\,y_i\,\bm x^\top \bar{\bm z}_i,
\qquad i=1,\ldots,n,
\]
and write $F(\bm x,\omega)$ as the random variable taking values $\{F_i(\bm x)\}_{i=1}^n$ under the baseline
distribution $\bm p = (p_1,\ldots, p_n)$ (in our experiments we take $p_i=1/n$).
To promote robustness to misclassified or low-margin points, we replace the mean loss with the tail risk
$\CVaR_\alpha(F(\bm x,\omega))$. With $\ell_2$ regularization, the resulting CVaR-SVM model (cf. \cite{takedasvm}) is
\begin{equation}\label{eq:cvar_svm_primal}
\min_{\bm x\in\mathbb{R}^{d+1}}
\Big\{ f(\bm x):=
\CVaR_\alpha\big(F(\bm x,\omega)\big)
+\frac{\lambda}{2}\sum_{i-1}^dx_i^2
\Big\},
\end{equation}
where $\lambda>0$ controls regularization and is taken to be $10^{-3}$ in the numerical experiments.

\paragraph{Data} In our numerical experiments, we fix a decision vector $\bm x=(0,x_1,\ldots,x_d)\in\mathbb{R}^{d+1}$ with $d=85$ and generate $n=2\times 10^4$ feature vectors $\bm z_i\stackrel{\mathrm{iid}}{\sim}\mathcal{N}(\mathbf 0,I_d)$. 
Let $\bar{\bm z}_i\coloneqq (1, \bm z_i^\top)^\top\in\mathbb{R}^{d+1}$ denote the extended feature vector. 
Given $\bm x$ and $\bar{\bm z}_i$, the class label is assigned by
\[
y_i=\operatorname{sign}(\bm x^\top \bar{\bm z}_i).
\]
\subsubsection{Deterministic case} In this setting, we solve the CVaR-SVM problem \eqref{eq:cvar_svm_primal} using \Cref{alg:inexact-easi-det}, Shor's \(r\)-algorithm (implemented following the recommendations in \cite{Stetsyuk2017}), the baseline method \Cref{alg:BaselineDet}, and the variant corresponding to the classical proximal point method proposed in \cite{kouri2022primal}. We abbreviate the latter by \(\mathrm{PD}\), short for \emph{primal-dual}, following the terminology of \cite{kouri2022primal}. However, rather than implementing that method exactly as in the original paper, we employ an implementation consistent with \textsc{EASIeST}, where a quadratic prox term replaces the Bregman prox term with necessary adjustments for handling inexactness (cf. \Cref{rem:PG_stability_l2}). We denote by $f_j = f(\bm x^j)$ the objective value at iterate $j$, and as the optimal CVaR value $f^\star = f(\bm x^\star)$, we use the value reported by the PSG solver \textsc{van}~\cite{PSG} with precision parameter $7$.
\paragraph{Parameter settings for \Cref{alg:inexact-easi-det} and PD} We set $\gamma_0=1 (\gamma_0=10^{-4} \text{ for PD}), \ c_\gamma = 1.08, \ \varepsilon_g = 10^{-6}, \ \varepsilon_{TV} = 10^{-5}, \varepsilon_{q} = 10^{-10}, |\mathcal{B}|=n,  \ K=130,  \ J=600,$
$$\rho_k^2= \begin{cases*}
10^{-3}& \text{if}\;$\alpha$=0.9 \\
3\times10^{-4}& \text{if}\;$\alpha$=0.95 \\
10^{-4}& \text{if}\;$\alpha$=0.98 
\end{cases*}
$$
As the subproblem solver, we use Algorithm~2 (adaptive accelerated gradient descent heuristic) from \cite{malitsky2019adaptive}.

\begin{table}[h]
\scriptsize
\caption{Convergence in oracle calls at prescribed accuracy levels.}
\label{tab:convergence_calls_merged}
\centering
\begin{tabular}{lcccccccccccc}
\toprule
& \multicolumn{4}{c}{$\alpha=0.90$} & \multicolumn{4}{c}{$\alpha=0.95$} & \multicolumn{4}{c}{$\alpha=0.98$} \\
\cmidrule(lr){2-5}\cmidrule(lr){6-9}\cmidrule(lr){10-13}
Gap
& \textsc{EASIeST} & $r$-alg & Baseline & PD
& \textsc{EASIeST} & $r$-alg & Baseline & PD
& \textsc{EASIeST} & $r$-alg & Baseline & PD \\
\midrule
$10^{-1}$ & 53  & 133 & 206 &62
         & 44  & 174 & 183 &26
         & 8   & 85 & 321 &103 \\
$10^{-2}$ & 80  & 182 & 539 &94
         & 117 & 306 & 335 &68
         & 14  & 330 & 444 &251 \\
$10^{-3}$ & 113 & 308 & 1454&137
         & 196 & 450 & 616&208
         & 71  & 473 & 880&510\\
$10^{-4}$ & 192 & 465 & $>$ 3000&273
         & 380 & 599 & 1486&447
         & 314 & 632 & ---&1045 \\
$10^{-5}$ & 423 & 659 & ---&494
         & 722 & 791 & ---&790
         & 602 & 762 & ---&--- \\
$10^{-6}$ & 707 & --- & ---&767
         & 1290 & --- & ---&987
         & 684 & 809 & ---&--- \\
\bottomrule
\end{tabular}
\end{table}
\begin{table}[h]
\scriptsize
\caption{Convergence in (per-scenario) function evaluations ($\times 10^{6}$) at prescribed accuracy levels.}
\label{tab:convergence_fun_merged}
\centering
\begin{tabular}{lcccccccccccc}
\toprule
& \multicolumn{4}{c}{$\alpha=0.90$} & \multicolumn{4}{c}{$\alpha=0.95$} & \multicolumn{4}{c}{$\alpha=0.98$} \\
\cmidrule(lr){2-5}\cmidrule(lr){6-9}\cmidrule(lr){10-13}
Gap
& \textsc{EASIeST} & $r$-alg & Baseline & PD
& \textsc{EASIeST} & $r$-alg & Baseline & PD
& \textsc{EASIeST} & $r$-alg & Baseline & PD \\
\midrule
$10^{-1}$ & 1.06  & 2.66 & 4.12 &1.24
         & 0.8  & 3.5 & 3.7 &0.52
         & 0.16   & 1.7 & 6.42&2.06 \\
$10^{-2}$ & 1.6  & 3.64 & 10.8&1.88
         & 2.3 & 6.1 & 6.7&1.36
         & 0.28  & 6.6 & 8.88&5.02 \\
$10^{-3}$ & 2.2 & 6.2 & 29&2.74
         & 3.7 & 9 & 12.3&4.16
         & 1.42  & 9.46 & 17.6&10.2 \\
$10^{-4}$ & 3.02 & 9.3 & ---&5.46
         & 5.07 & 11.98 & 29.7&8.94
         & 3.36 & 12.6 & ---&20.9 \\
$10^{-5}$ & 3.7 & 13.2 & ---&9.88
         & 5.73 & 15.82 & ---&15.8
         & 3.77 & 15.2 & ---&--- \\
$10^{-6}$ & 4.3 & --- & ---&15.34
         & 6.36 & --- & ---&19.74
         & 3.82 & 16.2 & ---&--- \\
\bottomrule
\end{tabular}
\end{table}

\begin{table}[h]
\scriptsize
\caption{Convergence in (per-scenario) gradient evaluations ($\times 10^6$) at prescribed accuracy levels.}
\label{tab:convergence_grad_merged}
\centering
\begin{tabular}{lcccccccccccc}
\toprule
& \multicolumn{4}{c}{$\alpha=0.90$} & \multicolumn{4}{c}{$\alpha=0.95$} & \multicolumn{4}{c}{$\alpha=0.98$} \\
\cmidrule(lr){2-5}\cmidrule(lr){6-9}\cmidrule(lr){10-13}
Gap
& \textsc{EASIeST} & $r$-alg & Baseline & PD
& \textsc{EASIeST} & $r$-alg & Baseline & PD
& \textsc{EASIeST} & $r$-alg & Baseline & PD \\
\midrule
$10^{-1}$ & 0.39  & 0.27 & 0.41&0.18
         & 0.34  & 0.17 & 0.18&0.09
         & 0.09   & 0.03 & 0.13&0.21 \\
$10^{-2}$ & 0.48  & 0.36 & 1.08&0.25
         & 0.51 & 0.31 & 0.34&0.17
         & 0.17  & 0.13 & 0.18&0.31 \\
$10^{-3}$ & 0.56 & 0.61 & 2.9&0.34
         & 0.64 & 0.45 & 0.62&0.34
         & 0.32  & 0.19 & 0.35&0.44 \\
$10^{-4}$ & 0.74 & 0.93 & ---&0.62
         & 0.85 & 0.6 & 1.47&0.59
         & 0.48 & 0.25 & ---&0.69 \\
$10^{-5}$ & 1.21 & 1.32 & ---&1.07
         & 1.21 & 0.79 & ---&0.95
         & 0.61 & 0.3 & ---&--- \\
$10^{-6}$ & 1.79 & --- & ---&1.63
         & 1.8 & --- & ---&1.15
         & 0.65 & 0.32 & ---&--- \\
\bottomrule
\end{tabular}
\end{table}

Across all three risk levels $\alpha\in\{0.90,0.95,0.98\}$, \Cref{tab:convergence_calls_merged}
shows that \textsc{EASIeST}, Shor's $r$-algorithm, and PD all substantially outperform the Baseline
method in terms of oracle calls. Overall, \textsc{EASIeST} is the most robust method, being the only
one that reaches gap $10^{-6}$ for all three values of $\alpha$. For $\alpha=0.90$,
\textsc{EASIeST} is the best performer across all reported gap levels, requiring, for instance,
$423$ oracle calls to reach gap $10^{-5}$ and $707$ calls to reach gap $10^{-6}$, compared with
$494$ and $767$ for PD, respectively. For $\alpha=0.95$, PD is particularly competitive: it is the
best performer at gaps $10^{-1}$ and $10^{-2}$, requiring only $26$ and $68$ oracle calls,
respectively, compared with $44$ and $117$ for \textsc{EASIeST}, and it also reaches gap $10^{-6}$
in $987$ calls, improving on the $1290$ calls required by \textsc{EASIeST}. At the same time,
\textsc{EASIeST} remains slightly better at the intermediate tighter levels $10^{-3}$--$10^{-5}$.
For the most risk-averse case\footnote{We remark that for $\alpha=0.98$ the PSG solver \textsc{van}
used to compute $f^\star$ reported only a feasible solution at the specified precision level $7$,
which may affect the absolute accuracy of the reported gaps in that regime.} $\alpha=0.98$,
\textsc{EASIeST} clearly outperforms the competing methods across all reported gap levels, reaching
gap $10^{-5}$ in $602$ oracle calls and gap $10^{-6}$ in $684$ calls, compared with $762$ and $809$
for the $r$-algorithm, while PD does not reach these tighter accuracies within the allotted budget.
Taken together, these results indicate that, in terms of oracle calls, all methods (except the baseline) perform at a broadly comparable level, with differences depending on the regime rather than revealing a uniformly superior method.

\Cref{tab:convergence_fun_merged,tab:convergence_grad_merged} provide a more detailed view of this
comparison in terms of cumulative (per-scenario) function and gradient evaluations (in units of
$10^6$). In terms of function evaluations, \textsc{EASIeST} is generally the most efficient method,
especially at tighter accuracies. For $\alpha=0.90$, it requires fewer function evaluations than
both PD and Shor's $r$-algorithm at every reported gap level; for example, at gap $10^{-5}$ the
counts are $3.7$ for \textsc{EASIeST}, $9.88$ for PD, and $13.2$ for the $r$-algorithm, and at
gap $10^{-6}$ they are $4.3$ and $15.34$ for \textsc{EASIeST} and PD, respectively. A similar
pattern is observed for $\alpha=0.98$, where \textsc{EASIeST} substantially improves on the
$r$-algorithm at the tightest reported levels (e.g., $3.77$ vs.\ $15.2$ at gap $10^{-5}$, and
$3.82$ vs.\ $16.2$ at gap $10^{-6}$). The main exception occurs for $\alpha=0.95$ at coarse
tolerances, where PD requires fewer function evaluations than \textsc{EASIeST} at gaps $10^{-1}$
and $10^{-2}$ ($0.52$ vs.\ $0.8$, and $1.36$ vs.\ $2.3$), although \textsc{EASIeST} becomes more
efficient from gap $10^{-3}$ onward. The lower function-evaluation counts of \textsc{EASIeST} are also partly explained by its block structure. Even in the present experiment, where the full block is used, \textsc{EASIeST}
employs block sampling \emph{without replacement} according to the current dual
probability vector $\bm q^k$. As the algorithm progresses, the probabilities
$q_i^k$ associated with non-tail scenarios become small and, in
practice, these scenarios are no longer sampled. Consequently, they stop
contributing to the cumulative function-evaluation count, whereas the competing
methods continue to evaluate the full sample at each iteration. Therefore, the savings observed in \Cref{tab:convergence_fun_merged} reflect not only convergence in oracle calls, but also the reduced effective evaluation cost induced by the sampling mechanism.

The gradient-evaluation counts present a more nuanced picture. PD is often quite competitive in this
metric and, for $\alpha=0.90$ and $\alpha=0.95$, it uses fewer gradient evaluations than
\textsc{EASIeST} at many of the reported gap levels; for instance, at $\alpha=0.95$ and gap
$10^{-6}$ the counts are $1.15$ for PD and $1.8$ for \textsc{EASIeST}. On the other hand, for
$\alpha=0.98$ the $r$-algorithm is the most economical in gradients, and \textsc{EASIeST} still
improves on PD whenever both methods reach the same target level. Thus, the tables suggest a clear
distinction between the two performance measures: \textsc{EASIeST} tends to be more efficient in
function evaluations and more reliable at tighter accuracies, whereas PD is often competitive in
gradient evaluations and can be very effective in oracle calls on some instances.

\subsubsection{Stochastic case}  
This section solves the CVAR-SVM problem \eqref{eq:cvar_svm_primal} using \Cref{alg:inexact-easi-stoch} and the baseline \Cref{alg:BaselineStoch}. Analogously to the previous section, as the optimal CVaR value $f^* = f(\bm x^*)$, we use the value reported by the PSG solver \textsc{van} with precision $7$. Note that, in the stochastic setting, PD does not admit the block-sampling mechanism used by \textsc{EASIeST}. The reason is that the iterates \(\bm{q}^k\) no longer remain in the interior of \(\mathcal{Q}\) (cf.\ \Cref{rem:PD_block_sampling}). Therefore, we exclude PD from the stochastic experiments.

To facilitate a fair comparison, we consider a \emph{constant} regime, where both \Cref{alg:inexact-easi-stoch} and the baseline \Cref{alg:BaselineStoch} are executed with the same fixed stepsize. 
In both cases, we fix $\alpha = 0.95$ and define
$
N_\alpha := (1-\alpha)n.
$
Both algorithms are then run with mini-batch size $\theta N_\alpha$, with
$
\theta \in \{0.5,0.6,0.7,0.8,0.9,1\}.
$
In this way, the batch size is scaled relative to the effective tail sample size.

The constant stepsize used in the experiments was chosen according to the standard stochastic subgradient scaling (cf. the notation in \Cref{th:convergencesubprob})
\begin{equation}\label{eq:constant_stepsize}
\beta  =  \frac{1}{C}\frac{R}{B\sqrt{T}}, \quad C \in [2,3]    
\end{equation}
where \(R\) is a problem-dependent distance scale, \(T\) is the prescribed iteration budget, and \(B\) is a bound on the second moment of the stochastic subgradients. In our implementation, we first computed \(R = \|x^0 - x^\star\|_2\) and fixed the total number of iterations to \(T=4400\). We then estimated \(B\) empirically by means of a pilot run of the baseline algorithm: specifically, we ran the baseline method for \(1000\) iterations with constant stepsize \(\beta = 0.1\) and batch size \(N_{\alpha}= (1-\alpha)n=1000\), and computed \(B\) as the root mean square of the observed stochastic subgradient norms over this run. Substituting the resulting estimate of \(B\) into the above expression yielded the constant stepsize used in the reported experiments. The constant \(C \in [2,3]\) is a stability factor introduced to make the nominal stepsize sufficiently conservative for use with smaller batch sizes.
\paragraph{Parameter settings for \Cref{alg:inexact-easi-stoch}} We set $\gamma_0=1, \ c_\gamma = 1.08, \ \varepsilon_g = 10^{-6}, \ \varepsilon_{TV} = 10^{-5}, \varepsilon_{q} = 10^{-10}, |\mathcal{B}|=8500,  \ K=20,  \ J=220, \ m = \theta N_\alpha, \ \rho_k^2 = 10^{-5}$.
As the subproblem solver, we use \eqref{eq:tupdate}--\eqref{eq:xupdate} with $\beta_j = \beta$ defined in \eqref{eq:constant_stepsize}. 

The \Cref{fig:stoch_const} shows that \textsc{EASIeST} consistently achieves a faster decay of the suboptimality gap and attains a lower final objective gap than the baseline method for all tested batch sizes. The gain is especially pronounced for larger batches, indicating that the method is particularly effective at exploiting informative tail samples.

However, this improved convergence behavior is accompanied by a higher computational cost in terms of function and gradient evaluations. Since \textsc{EASIeST} allocates a larger fraction of samples to the tail, each iteration becomes more expensive than in the baseline scheme. Consequently, the improvement observed in oracle-call complexity should be interpreted together with this increased evaluation cost. This tradeoff also suggests a natural improvement to the method in future work: \emph{adaptive batch-size} selection. Such a mechanism could reduce computational effort in the early stages of the algorithm and reserve larger, more accurate batches for later iterations, where precise tail estimation becomes more critical.

The \Cref{fig:block} indicates that block size $|\mathcal{B}|$ significantly affects convergence. While the method converges for block sizes close to $n$, reducing the block size initially accelerates convergence, with the most favorable behavior occurring around $|\mathcal B| \approx 0.4n$. In contrast, excessively small blocks lead to noticeably slower convergence. This behavior suggests that block size should not be chosen statically, and motivates the development of \emph{adaptive block-size} selection rules that could further improve the practical efficiency of the method. This observation also explains our choice of $|\mathcal B|=8500$ in the experiments reported in \Cref{fig:stoch_const}, as this value corresponds to an intermediate block size that provides a favorable trade-off between convergence speed and computational efficiency.
\begin{figure}[h]
    \centering
    \includegraphics[width=0.95\linewidth]{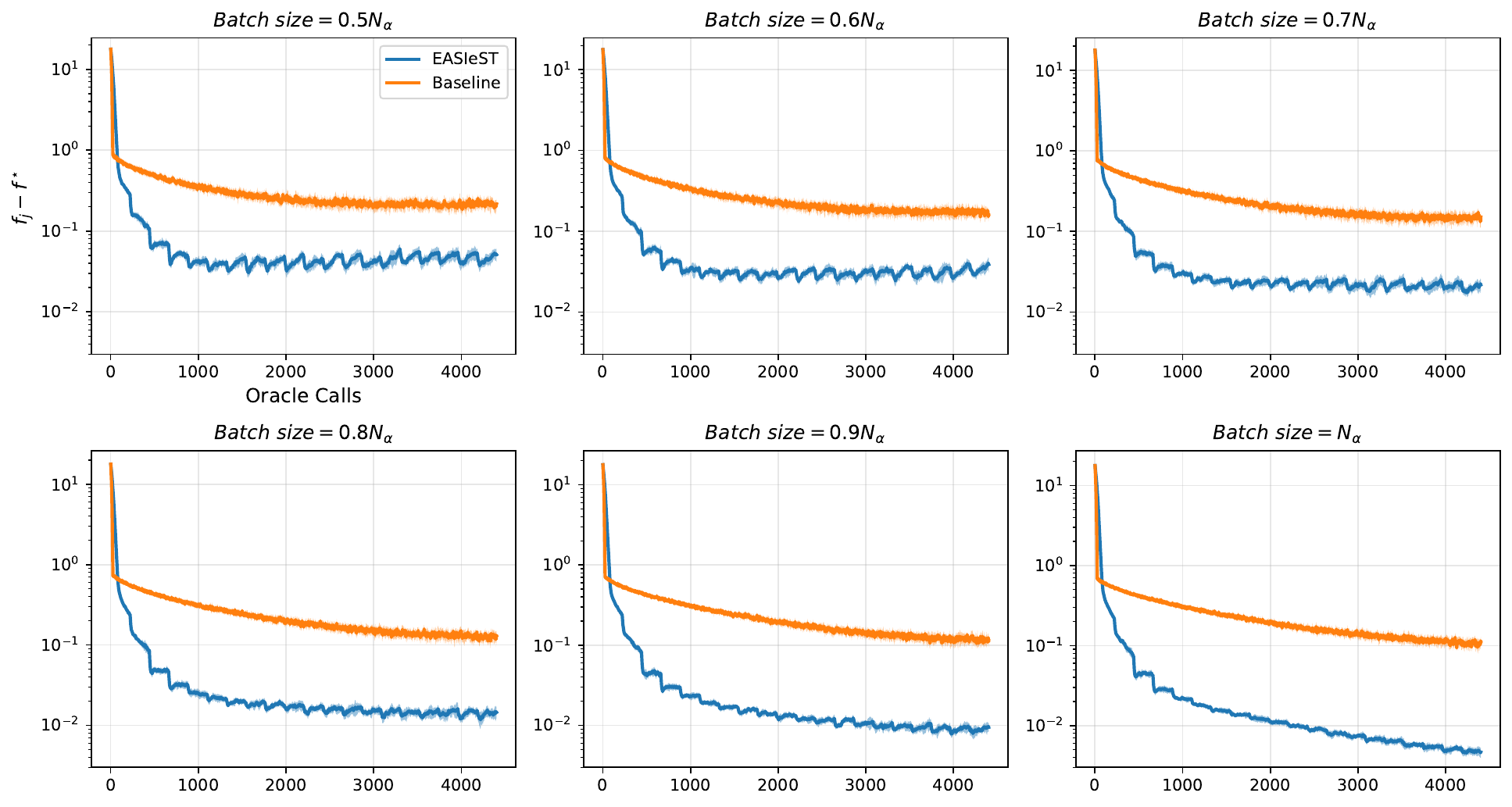}
    \caption{Comparison of the average suboptimality gap \(f_j-f^\star\), averaged over \(10\) independent runs, versus oracle calls for ``stochastic'' \textsc{EASIeST} and the baseline method under different batch sizes, ranging from \(0.5N_{\alpha}\) to \(N_{\alpha}, \ \alpha=0.95\).}
    \label{fig:stoch_const}
\end{figure}

\begin{figure}[h]
    \centering
    \includegraphics[width=0.6\linewidth]{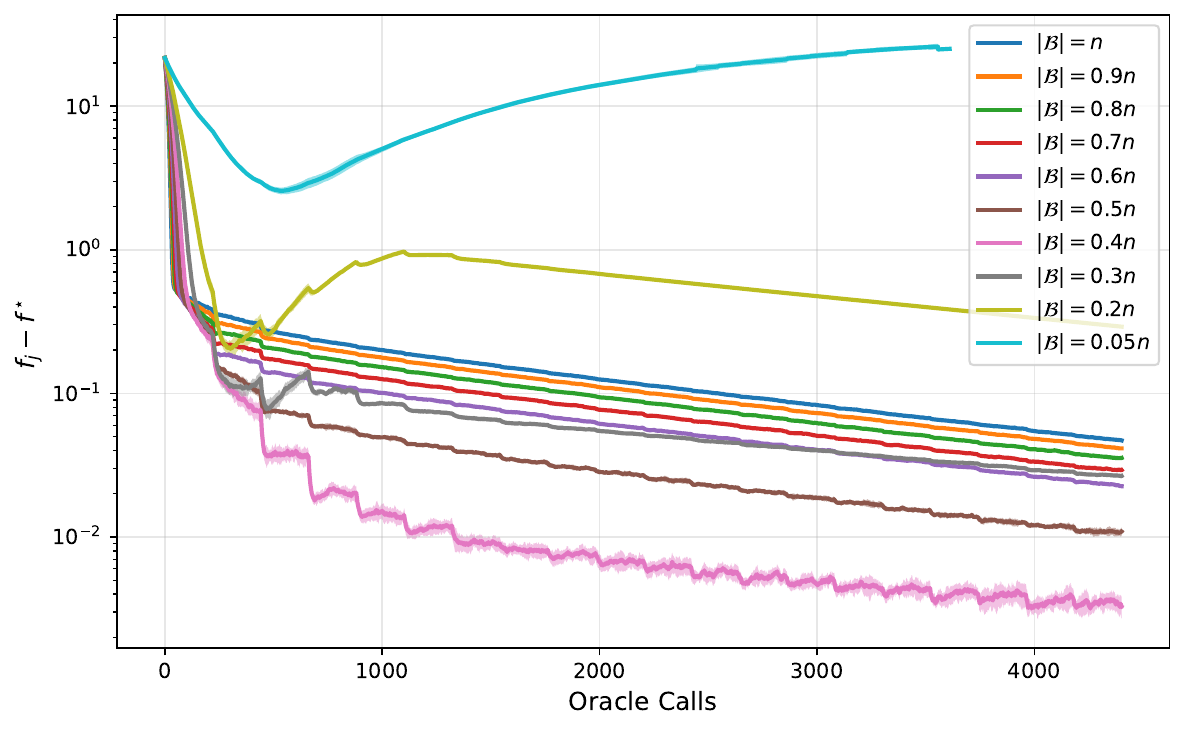}
    \caption{Effect of the block size $|\mathcal B|$ on the convergence of \Cref{alg:inexact-easi-stoch}. The figure reports the  suboptimality gap $f_j-f^*$ averaged over $10$ independent runs, versus oracle calls for for ``stochastic'' \textsc{EASIeST} with block sizes $|\mathcal B| \in \{n,0.9n,0.8n,0.7n,0.6n, 0.5n, 0.4n, 0.3n, 0.2n, (1-\alpha)n\}$. In all experiments, the mini-batch size is fixed at $N_\alpha, \ \alpha = 0.95$. }
    \label{fig:block}
\end{figure}
A possible explanation for this behavior is that moderate block sizes provide a better balance between adaptivity and stability in the update of the sampling distribution. When $|\mathcal B|$ is close to $n$, the update remains stable but may be overly conservative, since it distributes the correction across many scenarios, including those that are only weakly relevant to the active CVaR tail. In contrast, a moderately smaller block allows the method to adjust the distribution more selectively toward informative tail scenarios, which can accelerate convergence. If the block becomes too small, however, the update becomes excessively local and noisy, resulting in slower overall progress.
\subsection{Minimum-CVaR portfolio Optimization}\label{subsec:port}
We consider a portfolio of $d$ assets and a finite set of return scenarios
\(\{\bm z_i\}_{i=1}^n\), where each $\bm z_i\in\mathbb{R}^d$ represents the vector of asset returns in scenario $i$.
Let
\[
\bm x \coloneqq (x_1,\ldots,x_d)^\top\in\mathbb{R}^d
\]
denote the portfolio allocation vector, where $x_j$ is the proportion of wealth invested in asset $j$.
We restrict $\bm x$ to the simplex
\[
\cX \coloneqq \Big\{\bm x\in\mathbb{R}^d_+:\ \bm{1}^\top\bm{x} = 1\Big\},
\]
so that short selling is excluded and the budget constraint is satisfied.

For each scenario $i$, we define the portfolio loss
\[
F_i(\bm x)\coloneqq -\,\bm x^\top \bm z_i,
\qquad i=1,\ldots,n,
\]
and write $F(\bm x,\omega)$ as the random variable taking values $\{F_i(\bm x)\}_{i=1}^n$ under the baseline
distribution $\bm p=(p_1,\ldots,p_n)$ (in our experiments we take $p_i=1/n$).
To control downside risk, we minimize the conditional value-at-risk of the portfolio loss. This leads to the
minimum-CVaR portfolio optimization problem
\begin{equation}\label{eq:min_cvar_port_primal}
\min_{\bm x\in\cX}
\Big\{ f(\bm x):=
\CVaR_\alpha\big(F(\bm x,\omega)\big)
\Big\}.
\end{equation}
Unlike the classical mean--CVaR model, formulation \eqref{eq:min_cvar_port_primal} does not impose a lower bound on the expected portfolio return; instead, it seeks a portfolio that minimizes tail risk subject only to the budget and nonnegativity constraints.

\paragraph{Data}
In our numerical experiments, we fix the number of assets at $p$ and generate $n$ return scenarios $\bm z_i\in\mathbb{R}^p$, $i=1,\ldots,n$, independently from a multivariate normal distribution
\[
\bm z_i \stackrel{\mathrm{iid}}{\sim} \mathcal{N}(\bm \mu,\Sigma).
\]
The covariance matrix is constructed as
\[
\Sigma = A^\top A,
\]
where $A\in\mathbb{R}^{p\times p}$ is a random matrix with entries sampled independently from the uniform distribution on $[0,1]$. The mean vector $\bm \mu\in\mathbb{R}^p$ is generated componentwise from the uniform distribution on $[-0.1,10]$. The resulting sample
\(
\{\bm z_i\}_{i=1}^n
\)
is then used to define the portfolio loss scenarios in \eqref{eq:min_cvar_port_primal}.
\subsubsection{Deterministic case} This subsection studies the deterministic minimum-CVaR portfolio optimization problem. This experiment has two objectives. First, it demonstrates the efficiency of the proposed method in the presence of explicit constraints. Second, it compares \textsc{EASIeST} with a block version of the classical
proximal-point variant inspired by \cite{kouri2022primal}, obtained by replacing
the Bregman divergence $D_\psi(\bm q,\bm r)$ with
$\frac12\|\bm q-\bm r\|_2^2$ while imposing the same block-sampling structure
used by \textsc{EASIeST} in the full-block regime. We refer to this method as
block PD (BPD).

As the previous experiments show, the two methods are broadly comparable in terms of oracle calls. Nevertheless, PD attains this performance only at the cost of full function evaluations. In contrast, \textsc{EASIeST} retains the computational advantage of its block structure even in the full-block case: once the probabilities of certain scenarios fall to the level of machine precision, these scenarios are no longer sampled and thus stop contributing to function evaluations. PD does not possess this property. The experiment below shows that, in such a setting, BPD is not practically viable. Although this conclusion is consistent with the underlying theory, it is also important to illustrate it numerically.
As the optimal CVaR value $f^\star = f(\bm x^\star)$, we use the value reported by the PSG solver \textsc{van}~\cite{PSG} with precision parameter $10$.
\paragraph{Parameter settings for \Cref{alg:inexact-easi-det} and BPD} We set $\gamma_0=1 (\gamma_0=10^{-4} \text{ for PD}), \ c_\gamma = 1.08, \ \varepsilon_g = 10^{-6}, \ \varepsilon_{TV} = 10^{-6}, \varepsilon_{q} = 10^{-10}, |\mathcal{B}|=n,  \ K=100,  \ J=500, \ \rho_k = 10^{-4}$.
As the subproblem solver, we use Algorithm~3 (adaptive proximal gradient method) from \cite{malitsky2024adaptive}.

\begin{table}[h]
\scriptsize
\caption{Performance of BPD and \textsc{EASIeST} on the deterministic minimum-CVaR portfolio optimization problem for $\alpha=0.99$ with $p=100$ assets.}
\label{tab:port_det_pd_easi}
\centering
\begin{tabular}{ccccc}
\toprule
& \multicolumn{2}{c}{BPD} & \multicolumn{2}{c}{\textsc{EASIeST}} \\
\cmidrule(lr){2-3}\cmidrule(lr){4-5}
$n$ & Best Gap & Total Calls & Best Gap &  Total Calls \\
\midrule
1000   & $2.38 \times 10^{-3}$ &  659 & $4.82 \times 10^{-9}$ &  1509 \\
10000  & $1.69 \times 10^{-6}$ &  790 & $4.13 \times 10^{-7}$  & 892  \\
50000  & $1.68 \times 10^{-2}$ & 47  &  $1.44 \times 10^{-9}$ &  563  \\
100000 & $2.33 \times 10^{-2}$ & 52  & $5.18 \times 10^{-9}$ & 502  \\
\bottomrule
\end{tabular}
\end{table}

Table~\ref{tab:port_det_pd_easi} illustrates a limitation of BPD in the block-sampling setting. This behavior is consistent with the convergence theory developed in \Cref{th:blockproxcovfvalues,th:blockproxconvergence}, where the block-sampling distribution is required to satisfy $\pi_i^k>0$ for all scenarios $i$ and all iterations $k$. In the case of \textsc{EASIeST}, this condition is naturally satisfied by taking $\pi_i^k=q_i^k$, since the Bregman update preserves strict positivity of the dual weights. By contrast, in BPD the quadratic-prox update involves an $\ell_2$-projection, which does not guarantee positivity; i.e., some weights may be projected exactly to zero. Once a scenario receives zero weight, it is excluded from subsequent sampling, so the method can no longer revisit or correct it. This violates the requirement underlying the convergence theory and explains why BPD cannot be reliably implemented with the same block-sampling mechanism.

In contrast, \textsc{EASIeST} remains compatible with the block-sampling framework. Since its weights stay strictly positive, the induced sampling distribution remains well-defined throughout the run and continues to satisfy the assumptions of \Cref{th:blockproxcovfvalues,th:blockproxconvergence}. At the same time, scenarios whose probabilities become negligible up to machine precision effectively cease to contribute to the function-evaluation cost, so \textsc{EASIeST} retains the computational advantages of block sampling without losing convergence. This distinction is clearly reflected in \cref{tab:port_det_pd_easi}: \textsc{EASIeST} continues to attain high accuracy across all tested values of $n$, whereas BPD stagnates at relatively large gaps when block sampling is imposed.



\section{Conclusion}\label{sec:Conclusion}
We presented \textsc{EASIeST}, a new framework for CVaR optimization based on the dual representation of CVaR and a Bregman proximal point scheme on its risk envelope. The proposed approach combines two desirable features within a single algorithmic mechanism: adaptive smoothing of the nonsmooth CVaR objective and adaptive importance sampling that progressively concentrates the sampling distribution on tail scenarios while preserving interior feasibility. This yields a method that is both theoretically grounded and practically tailored to the structure of CVaR problems.

On the theoretical side, we derived the method from a saddle-point formulation of CVaR and showed that the generalized Fermi--Dirac entropy induces a natural geometry for the dual probability updates. This geometry leads to closed-form block Bregman proximal steps, keeps the iterates in the interior of the feasible risk envelope, and provides a built-in tail-learning mechanism. For convex problems, we established convergence guarantees for the inner stochastic subproblem solver, convergence in function values for the outer method, and almost sure convergence of the dual iterates under exact subproblem solutions.

Our numerical experiments demonstrate that the proposed framework is effective across representative applications in machine learning and quantitative finance. 
In the deterministic support-vector classification benchmark, \textsc{EASIeST} consistently outperformed the baseline method and was broadly competitive with Shor's $r$-algorithm in oracle calls, while requiring substantially fewer function evaluations at tighter accuracies. In the stochastic setting, \textsc{EASIeST} exhibited faster decay of the suboptimality gap than the baseline across all tested batch sizes, with especially strong performance for larger batches. At the same time, the experiments highlight an important tradeoff: the improved tail-focused sampling may increase per-iteration function and gradient costs, reflecting the price of the smoothing and importance-sampling mechanism.

These observations point to several promising directions for future work. In particular, the stochastic experiments suggest that adaptive batch-size and block-size selection rules could further improve practical efficiency by balancing early-stage exploration with later-stage tail refinement; cf.~\cite{beiser2023adaptive}.

\appendix
\section{Proof of Theorem~\ref{thm:BregmanConvergence}}
\label{app:thm_bregman_convergence}

\begin{proof}
We split the proof into two steps.

\smallskip
\noindent\textsl{Step 1.} Since $D_\psi(\bm{q}^{k+1},\bm{q}^k)\ge 0$, we have
\[
\varphi(\bm{q}^{k+1})
\le
\varphi(\bm{q}^{k+1})+\frac{1}{\gamma_k}D_\psi(\bm{q}^{k+1},\bm{q}^k).
\]
By optimality of $\bm{q}^{k+1}=\prox_{\gamma_k\varphi}(\bm{q}^k)$ in~\eqref{eq:ProxOperator},
\[
\varphi(\bm{q}^{k+1})+\frac{1}{\gamma_k}D_\psi(\bm{q}^{k+1},\bm{q}^k)
\le
\varphi(\bm{q}^{k})+\frac{1}{\gamma_k}D_\psi(\bm{q}^{k},\bm{q}^k)
=
\varphi(\bm{q}^{k}),
\]
hence $\varphi(\bm{q}^{k+1})\le \varphi(\bm{q}^k)$.

\smallskip
\noindent\textsl{Step 2.} Using the three-point identity \cite{chen1993convergence},
\[
D_\psi(\bm{s},\bm{q})+D_\psi(\bm{q},\bm{r})
=
D_\psi(\bm{s},\bm{r})
+
\big\langle \nabla\psi(\bm{r})-\nabla\psi(\bm{q}),\,\bm{s}-\bm{q}\big\rangle,
\]
and the first-order optimality condition for~\eqref{eq:ProxOperator},
\[
\big\langle \gamma g_\varphi(\bm{q})+\nabla\psi(\bm{q})-\nabla\psi(\bm{r}),\,\bm{s}-\bm{q}\big\rangle\ge 0
\quad
\text{for all }\bm{s}\in\mathcal{Q},
\]
for some $g_\varphi(\bm{q})\in\partial\varphi(\bm{q})$, we set
$\gamma=\gamma_k$, $\bm{q}=\bm{q}^{k+1}$, $\bm{r}=\bm{q}^k$, $\bm{s}=\bm{q}^\star$ to obtain
\[
D_\psi(\bm{q}^\star,\bm{q}^{k+1})+D_\psi(\bm{q}^{k+1},\bm{q}^k)
\le
D_\psi(\bm{q}^\star,\bm{q}^{k})
+
\gamma_k\langle g_\varphi(\bm{q}^{k+1}),\,\bm{q}^\star-\bm{q}^{k+1}\rangle.
\]
By convexity of $\varphi$,
$\langle g_\varphi(\bm{q}^{k+1}),\,\bm{q}^\star-\bm{q}^{k+1}\rangle
\le
\varphi(\bm{q}^\star)-\varphi(\bm{q}^{k+1})$.
Dropping the term $D_\psi(\bm{q}^{k+1},\bm{q}^k)\ge 0$ gives
\[
D_\psi(\bm{q}^\star,\bm{q}^{k+1})-D_\psi(\bm{q}^\star,\bm{q}^{k})
\le
\gamma_k\big(\varphi(\bm{q}^\star)-\varphi(\bm{q}^{k+1})\big).
\]
Summing from $k=0$ to $N-1$ and using Step 1 (non-increasing $\varphi(\bm{q}^k)$) yields
\[
D_\psi(\bm{q}^\star,\bm{q}^{N})-D_\psi(\bm{q}^\star,\bm{q}^{0})
\le
\big(\varphi(\bm{q}^\star)-\varphi(\bm{q}^{N})\big)\sum_{k=0}^{N-1}\gamma_k.
\]
Rearranging and using $D_\psi(\bm{q}^\star,\bm{q}^{N})\ge 0$ proves~\eqref{eq:BregmanConvergence}.
\end{proof}
\section{Baseline algorithms and CVaR subgradient constructions}\label{app:baselines}
For completeness, we collect here the baseline methods used in the numerical experiments together with the corresponding CVaR subgradient constructions.

\subsection{Deterministic baseline}\label{app:det_baseline}

For a given $\bm{x} \in \mathcal{X}$, let $t^\star(\bm{x}) = \var_\alpha(F(\bm{x},\omega))$ and recall that $F(\bm{x},\omega)$ is represented by the $n$-dimensional vector $(F_1(\bm{x}),\ldots, F_n(\bm{x}))$. Define
\[
I_{\diamond}(\bm x):=\{\,i:\ F_i(\bm x)\ \diamond\ t^\star(\bm x)\,\},
\qquad \diamond\in\{>,<,=\}
\]
and
\[
q_i(\bm{x}) \;=\;
\begin{cases}
\dfrac{p_i}{1-\alpha}, & i\in I_{>}(\bm{x}),\\[6pt]
\dfrac{\theta_i\,p_i}{1-\alpha}, & i\in I_{=}(\bm{x}),\\[6pt]
0, & i\in I_{<}(\bm{x}),
\end{cases}
\qquad\text{so that}\qquad
\sum_{i=1}^n q_i(\bm{x}) = 1.
\]
Then
\begin{equation}\label{eq:CVaR_subgrad}
 \partial_{\bm{x}} \cvar_\alpha(F(\bm{x},\omega))
 :=
 \Big\{
 g(\bm{x}) = \sum_{i=1}^n q_i(\bm{x})g_{F_i(\bm{x})}
 \ \text{for all feasible choices of $0 \leq \theta_i \leq 1$}
 \Big\},
\end{equation}
where $g_{F_i(\bm{x})} \in \partial F_i(\bm{x})$.

\begin{algorithm2e}[h]
\DontPrintSemicolon
\caption{\label{alg:BaselineDet}The deterministic baseline.}
\SetKwInOut{Input}{Input}
\SetKwInOut{Output}{Output}
\Input{$\bm{x}^0,\bm{x}^{-1} \in \mathcal{X}, \ u>1, \ d\in (0,1), \ \beta_0, \varepsilon >0, \ K \in \mathbb{N}$.}
\Output{an approximate solution $\mathbold{x}^\star\in\mathcal{X}$.}
\For{$k=0,1,2,\ldots,K$}{
\textit{Step 1.} Compute $g(\bm{x}^k) \in \partial_{\bm{x}} \cvar_\alpha(F(\bm{x}^k,\omega))$; \quad (cf. \eqref{eq:CVaR_subgrad})\;

\BlankLine
\textit{Step 2.} Compute $Q_k = \beta_k \|g(\bm{x}^k)\|_2$;\;

\BlankLine
\If(\tcp*[f]{termination}){$Q_k \leq \varepsilon$}{
      \textbf{break}
    }\;

\BlankLine
\textit{Step 3.} Compute $T_k = \langle g(\bm{x}^k),\bm{x}^{k-1} - \bm{x}^k \rangle$;\;

\BlankLine
\textit{Step 4.} Adjust the stepsize
\[
\beta_{k+1}=
\begin{cases}
u \beta_k, &\text{if } T_k>0,\\[4pt]
d \beta_k, &\text{if } T_k\le 0.
\end{cases}
\]
\;

\BlankLine
\textit{Step 5.} Find the next approximation
\[
\bm{x}^{k+1} = \operatorname{proj}_{\mathcal{X}}\Big[\bm{x}^{k} - \beta_kg(\bm{x}^k)\Big].
\]
\;
}
\Return $\bm x^\star \leftarrow \bm x^{k}, \ \cvar_\alpha(F(\bm x^\star,\omega))$\;
\end{algorithm2e}

\subsection{Stochastic baseline}\label{app:stoch_baseline}
\begin{algorithm2e}[h!]
\DontPrintSemicolon
\caption{\label{alg:BaselineStoch}The stochastic baseline.}
\SetKwInOut{Input}{Input}
\SetKwInOut{Output}{Output}
\Input{$\bm{x}^0,\bm{x}^{-1} \in \mathcal{X}, \ u>1, \ d,D\in (0,1), \ \beta_0, \varepsilon >0, \ K \in \mathbb{N}, \ A_0=z_0=0$.}
\Output{an approximate solution $\mathbold{x}^\star\in\mathcal{X}$.}
\For{$k=0,1,2,\ldots,K$}{
\emph{Step 1.} Sample indices $\xi_1,\ldots,\xi_m \stackrel{\mathrm{iid}}{\sim} \bm p$;

\BlankLine
\textit{Step 2.} Compute $g(\bm{x}^k; \xi) = \sum_{j=1}^m \hat q_j(\bm x^k)\,g(\bm x^k;\xi_j);$ \quad (cf. \eqref{eq:CVaR_sub_stoch})\;

\BlankLine
\emph{Step 3.} Average the norm: $A_{k+1} = A_k + \bigl(\|g(\bm{x}^k; \xi)\|_2-A_k\bigr)D$ and set $Q_k = \beta_kA_{k+1};$

\BlankLine
\If(\tcp*[f]{termination}){$Q_k \leq \varepsilon$}{
      \textbf{break}
    }\;

\BlankLine
\textit{Step 4.} Compute $T_k = \langle g(\bm{x}^k;\xi),\bm{x}^{k-1} - \bm{x}^k \rangle$;\;

\BlankLine
\emph{Step 5.} Average the inner product: $z_{k+1} = z_k + (T_k-z_k)D;$

\BlankLine
\textit{Step 6.} Adjust the stepsize
\[
\beta_{k+1}=
\begin{cases}
u \beta_k, &\text{if } z_{k+1}>0,\\[4pt]
d \beta_k, &\text{if } z_{k+1}\le 0.
\end{cases}
\]
\;

\BlankLine
\textit{Step 7.} Find the next approximation
\[
\bm{x}^{k+1} = \operatorname{proj}_{\mathcal{X}}\Big[\bm{x}^{k} - \beta_kg(\bm{x}^k; \xi)\Big].
\]
\;
}
\Return $\bm x^\star \leftarrow \bm x^{k}, \ \cvar_\alpha(F(\bm x^\star,\omega))$\;
\end{algorithm2e}
Let $\xi_1,\ldots,\xi_m \stackrel{\mathrm{iid}}{\sim} \bm p$ be scenario indices (sampled with replacement) and define
\[
F^{(m)}(\bm x)\coloneqq \big(F_{\xi_1}(\bm x),\ldots,F_{\xi_m}(\bm x)\big)^\top .
\]
Define the empirical $\alpha$-quantile
\[
t_m^\star(\bm x)\in\arg\min_{t\in\mathbb{R}}
\left\{
t+\frac{1}{(1-\alpha)m}\sum_{j=1}^m\big(F_{\xi_j}(\bm x)-t\big)_+
\right\},
\]
equivalently, $t_m^\star(\bm x)$ is an $\alpha$-quantile of $\{F_{\xi_j}(\bm x)\}_{j=1}^m$.
Define the batch index sets
\[
J_{\diamond}(\bm x)\coloneqq\big\{\,j\in\{1,\ldots,m\}:\ F_{\xi_j}(\bm x)\ \diamond\ t_m^\star(\bm x)\,\big\},
\qquad \diamond\in\{>,<,=\}.
\]
Set the batch weights
\[
\hat q_j(\bm x)=
\begin{cases}
\dfrac{1}{(1-\alpha)m}, & j\in J_{>}(\bm x),\\[6pt]
\dfrac{\theta_j}{(1-\alpha)m}, & j\in J_{=}(\bm x),\\[6pt]
0, & j\in J_{<}(\bm x),
\end{cases}
\qquad\text{so that}\qquad
\sum_{j=1}^m \hat q_j(\bm x)=1.
\]
Compute subgradients $g(\bm x;\xi_j)\in\partial F_{\xi_j}(\bm x)$. Then, a mini-batch stochastic CVaR subgradient is
\begin{equation}\label{eq:CVaR_sub_stoch}
g(\bm x; \xi)\in \partial_{\bm x}\,\cvar_\alpha^{(m)}\big(F(\bm x,\omega)\big),
\qquad
g(\bm x, \xi)=\sum_{j=1}^m \hat q_j(\bm x)\,g(\bm x;\xi_j),
\end{equation}
for any feasible choice of the weights $\{\theta_j\}_{j\in J_{=}(\bm x)}$.
\section{Subroutines}





The following subroutine \Cref{alg:block_prox_update} is used in both the stochastic and deterministic versions of the algorithm.

\begin{algorithm2e}[h!]
\DontPrintSemicolon
\caption{\label{alg:block_prox_update}\textsc{ProxUpdate}}
\SetKwInOut{Input}{Input}
\SetKwInOut{Define}{Define}
\SetKwInOut{Output}{Output}

\Input{$\bm x^{k,j}\in\cX$, reference $\bm s^k\in \mathbb{R}^n$ and $\bm q^k \in \mathcal{Q}$, block $\mathcal B_k\subseteq\{1,\ldots,n\}$, stepsize $\gamma_k>0$, block mass $\delta_k=\sum_{i\in\mathcal B_k}q_i^k$.}
\Define{
Sigmoid map: \( \sigma_i(s)\coloneqq \dfrac{p_i}{1-\alpha}\,\dfrac{e^{s}}{1+e^{s}} .\)
}
\Output{Updated logits $\bm s^{k,j+1}$ and weights $\bm q^{k,j+1}$.}

\BlankLine
Evaluate $\{F_i(\bm x^{k,j})\}_{i\in\mathcal B_k}$\;

\BlankLine
Set
\[
\tilde s^{k,j+1}_i=
\begin{cases}
s^{k}_i+\gamma_k F_i(\bm x^{k,j}), & i\in\mathcal B_k,\\
s^{k}_i, & i\notin\mathcal B_k.
\end{cases}
\]

\BlankLine
Find $t\in\mathbb R$ such that
\[
\sum_{i\in\mathcal B_k}\sigma_i\!\big(\tilde s_i^{k,j+1}+t\big)=\delta_k.
\]

\BlankLine
Set
\[
s^{k,j+1}_i=
\begin{cases}
\tilde s^{k,j+1}_i+t, & i\in\mathcal B_k,\\
s^{k}_i, & i\notin\mathcal B_k,
\end{cases}
\qquad
q^{k,j+1}_i=
\begin{cases}
\sigma_i(s_i^{k,j+1}), & i\in\mathcal B_k,\\
q_i^k, & i\notin\mathcal B_k.
\end{cases}
\]

\BlankLine
\Return $(\bm s^{k,j+1},\bm q^{k,j+1})$\;
\end{algorithm2e}
\phantomsection
\bibliographystyle{plain}
\newpage
\bibliography{main.bib}
\end{document}